\documentclass{amsart}
\usepackage{amsmath, amsthm, amsfonts, amssymb, amscd, hyperref, marginnote, todonotes, accents, color, multicol}
\usepackage{amsxtra}     

\usepackage{epsfig}
\usepackage{verbatim}
\usepackage[all,cmtip]{xy}
\usepackage{mathrsfs}
\usepackage{enumerate}
\allowdisplaybreaks
\usepackage[utf8]{inputenc}

\usepackage[margin=1in]{geometry}
\newcommand{\clC}[1]{\mathbf{C}(#1)}
\newcommand{\clS}{\mathbf{S}}
\newcommand{\clT}{\mathbf{T}}

\newcommand{\clB}{\mathbf{B}}

\newcommand{\clGS}{\mathbf{GS}}
\newcommand{\clGT}{\mathbf{GT}}
 \newcommand{\clGH}[1]{\mathbf{GH}(#1)}
\newcommand{\clG}[1]{\mathbf{G}(#1)}

\newcommand{\clH}[1]{\mathbf{H}(#1)}

\def\CA{{\mathcal A}}
\def\CB{{\mathcal B}}

\def\CP{{\mathcal P}}

\newcommand{\BQ}{{\mathbb Q}}
\newcommand{\BZ}{{\mathbb Z}}

\newcommand{\fm}{{\mathfrak m}}

\newcommand{\kk}{\mathsf{k}}

\newcommand{\Poi}[2]{\operatorname{P}^{#1}_{#2}}
\newcommand{\im}{\operatorname{Im}}

\newcommand{\HH}[2]{\operatorname{H}_{#1}(#2)}

\newcommand{\Span}{\operatorname{Span}}

\newcommand{\rank}{\operatorname{rank}}
\newcommand{\Ker}{\operatorname{Ker}}
\newcommand{\Coker}{\operatorname{Coker}}

\newcommand{\depth}[1]{\operatorname{depth}{#1}}

\newcommand{\Tor}[4]{\operatorname{Tor}_{#1}^{#2}(#3,#4)}

\newcommand{\codepth}{\operatorname{codepth}}

\theoremstyle{plain}
\newtheorem{theorem}{Theorem}[section]
\newtheorem{corollary}[theorem]{Corollary}
\newtheorem{lemma}[theorem]{Lemma}

\newtheorem{proposition}[theorem]{Proposition}

\theoremstyle{definition}
\newtheorem{examples}[theorem]{Examples}
\newtheorem{chunk}[theorem]{}

\newtheorem{remark}[theorem]{Remark}

\newtheorem{summary}[theorem]{Summary}
\newtheorem{construction}[theorem]{Construction}

\numberwithin{equation}{theorem}
\newtheorem{question}[theorem]{Question}

\begin{document}
\title{A truncated minimal free resolution of the residue field}
\author[V.~C.~Nguyen]{Van C.~Nguyen}
\address{Department of Mathematics, United States Naval Academy, Annapolis, MD 21402, U.S.A.}
\email{vnguyen@usna.edu}
\urladdr{https://sites.google.com/view/vcnguyen}
\author[O.~Veliche]{Oana Veliche}
\address{Department of Mathematics, Northeastern University, Boston, MA 02115, U.S.A.}
\email{o.veliche@northeastern.edu}
\urladdr{https://web.northeastern.edu/oveliche}
\date{\today}  


\begin{abstract}
In a paper in 1962, Golod proved that the Betti sequence of the residue field of a local ring attains an upper bound given by Serre if and only if the homology algebra of the Koszul complex of the ring has trivial multiplications and trivial Massey operations. This is the origin of the notion of Golod ring. Using the Koszul complex components he also constructed a minimal free resolution of the residue field. In this article, we extend this construction up to degree five for \emph{any} local ring. We describe how the multiplicative structure and the triple Massey products of the homology of the Koszul algebra are involved in this construction. As a consequence, we provide explicit formulas for the first six terms of  a sequence that measures how far the ring is from being Golod.
\end{abstract}

\maketitle
\thispagestyle{empty}


\section{Introduction}
\label{sec:Introduction}

Throughout the paper, $(R,\fm,\kk)$ is a  local noetherian ring with maximal ideal $\fm$ and residue field $\kk$, of embedding dimension $n$ and codepth $c=n-\depth{}R$.  Let $K^R$ be the Koszul complex  of $R$ on a minimal set of generators of  $\fm$.
Serre pointed out that there is always a coefficient-wise inequality between  the Poincar\'e series of the  $R$-module $\kk$ and  a series of rational form involving the ranks of the homologies of $K^R$:
 \begin{equation}
 \label{golodP}
  \Poi{R}{\kk}(t):=\sum_{i=0}^\infty\rank_{\kk} \Tor iR\kk\kk t^i \preccurlyeq \frac{(1+t)^n}{1 -\sum_{i=1}^c \rank_{\kk} \HH{i}{K^R} t^{i+1}}.
\end{equation}
In \cite{G}, Golod proved that a ring $R$ attains this upper bound if and only  if the graded-commutative algebra $\HH{}{K^R}$ has trivial multiplications and trivial Massey operations; such a ring is now called a {\it Golod ring}. In the same paper, he also constructed the minimal free resolution of the $R$-module $\kk$ in terms of $K^R$.

In \cite[Corollary 5.10]{Av74}, Avramov proved that the Poincar\'e series $\Poi{R}{\kk}(t)$ is completely determined by the structure of the Koszul homology algebra $A:=\HH{}{K^R}$ as an algebra with Massey operations. One can now ask the following question:
\begin{question}
\label{question} For \emph{any} local ring $(R,\fm,\kk)$, given the knowledge of its Koszul homology algebra $A$, with its products and Massey operations, how does one construct a \emph{minimal} free resolution of the residue field $\kk$?
\end{question}
In this paper, we answer this question explicitly up to degree five. For any local ring $R$, using the Koszul complex $K^R$ as building blocks and the graded-commutative structure of its homology $A$, we construct a  minimal free resolution of the $R$-module $\kk$, up to degree five, see Construction~\ref{construction} and Theorem \ref{main}:
\begin{equation*}
{F} :\qquad F_5\xrightarrow{\partial_5^F}F_4\xrightarrow{\partial_4^F}F_3\xrightarrow{\partial_3^F}F_2\xrightarrow{\partial_2^F}F_1\xrightarrow{\partial_1^F}F_0 \to \kk \to 0.
\end{equation*}

The higher degrees of the resolution can be extended similarly for some special cases of $R$, but in general it requires an understanding of higher Massey products, which remains elusive and is left for future projects. Indeed, Massey operations can be represented by using the $A_\infty$-algebra structure, which the Koszul homology algebra $A$ possesses, see for example, \cite{LPWZ} by Lu, Palmieri, Wu, and Zhang. Moreover, the $A_\infty$-structure was also used by Burke in \cite{B} to construct certain projective resolutions. These connections suggest that perhaps the $A_\infty$-algebra structure of $A$ may play a role in giving a complete answer to Question~\ref{question} from this perspective.

To prepare for Construction~\ref{construction}, in Section~\ref{sec multiplication} we analyze in detail the multiplicative structure of the algebra $A$, up to degree four, in terms of  bases of $A_i$ and  $A_i\cdot A_j$, and necessary maps. Moreover, in degree four  the Massey products appear for the first time, as  ternary Massey products  $\langle A_1,A_1,A_1\rangle$; we give a description of the elements of this set in Proposition~\ref{massey}. 

In Section~\ref{sec applications}, we obtain several direct applications of Theorem \ref{main}. In Corollary~\ref{betti}, we explicitly describe the Betti numbers $\beta_i:=\rank_{\kk} \Tor iR\kk\kk$, up to degree five, in terms of the multiplicative invariants of the Koszul homology algebra $A$. Let $a_i:=\rank_{\kk} A_i$ and consider the difference of series 
$$\CP(t):=\frac{(1+t)^n}{1 -\sum_{i=1}^c a_i t^{i+1}}-\Poi{R}{\kk}(t)$$ 
that measures how far the ring $R$ is from being Golod, that is, how far the Betti numbers of $R$ are from their maximum possible values. In Proposition~\ref{prop:P}, we compute the first six coefficients of $\CP(t)$ in terms of multiplicative invariants of $A$:
$$\CP_0=\CP_1=\CP_2=0,\quad  \CP_3=q_{11},\quad \CP_4=(n+1)q_{11}+q_{12},\quad \CP_5=\big(\textstyle{{n+1} \choose 2}+2a_1\big)q_{11}+(n+1)q_{12}+a-b,$$
where $q_{ij}=\rank_\kk(A_i \cdot A_j)$, and $a,b$ are described in Summary~\ref{summary}.  For any local ring $R$ of embedding dimension $n$ with rational Poincar\'e series of the form $\Poi{R}{\kk}(t)=(1+t)^n/d(t)$, we express the denominator, up to degree five, as follows:
$$d(t)=\big(1-\textstyle{\sum_{i=1}^c a_it^{i+1}}\big)+ q_{11}t^3+(q_{11}+q_{12})t^4+(q_{12}-b+a)t^5 +f(t)t^6,$$
for some $f(t)\in\BZ[t]$, see Proposition \ref{Rational Poincare}.  We consider a few classes of rings with such rational Poincar\'e series, see Corollary~\ref{yoshino} and Examples \ref{exp:yoshino} and \ref{roos}, and using the coefficients of the denominator we obtain algebraic invariants of their Koszul homology algebra $A$. The Poincar\'e series $\Poi{R}{\kk}(t)$ can be described by using the deviations $\varepsilon_i$'s, see for example \cite{Av1}. In Corollary~\ref{deviations}, we give a description of the first five deviations in terms of the algebraic invariants of $A$.

In Section \ref{sec: example}, we illustrate Construction~\ref{construction} through an example of a ring of codepth 4 examined by Avramov in \cite{Av74}. In that example, Avramov provided a nontrivial indecomposable Massey product element in $\langle A_1,A_1,A_1\rangle$  that does not come from  multiplications of the homology. Using Proposition~\ref{massey},  we prove that, up to a scalar, this is the only element with this property, modulo the products in homology.


\section{Multiplicative structure on the homology of the Koszul algebra}
\label{sec multiplication}

Let $(R,\fm,\kk)$ be a local ring of embedding dimension $n$ and codepth $c$. The differential graded algebra structure on  the Koszul complex $K = K^R$ on a minimal set of generators of the maximal ideal $\fm$ induces a graded-commutative algebra structure on the homology of $K$:
$$A:=\HH{}{K} =A_0 \oplus A_1 \oplus A_2 \oplus \dots \oplus A_c.$$
In this section we discuss this multiplicative structure up to degree four. We will use this structure extensively in constructing a truncated minimal free resolution of the residue field over $R$ in the next section.

The following notation is used throughout the paper: the differential map on $K$ is denoted by $\partial^K$, a homogeneous element  of degree $i$ in the Koszul complex $K$ is denoted by $\pi^i$, a representative in $K_i$  of an element in the homology  $A_i$ is denoted by $p^i$, and elements of $R$ are denoted by $\alpha$'s, $\beta$'s, and $\gamma$'s. The subscript of a homogeneous element indicates its index in a tuple and the superscript indicates its homological degree. Set
\begin{equation}
\label{a_i}
a_i:=\rank_{\kk} (A_i)\qquad\text{and}\qquad q_{ij}:= \rank _\kk (A_i \cdot A_j),\  \text{for all}\  0\leq i,j\leq c.
\end{equation}

It is clear that $A_0=\kk$.
For each $1\leq i \leq a_1$ we consider  $z_i^1\in\Ker\partial_1^ K$  such that
\begin{equation}
\label{z1}
\{[z^1_i]\}_{i=1,\dots ,a_1} \quad \text{is a  basis of}\ A_1.
\end{equation} 
In particular, every element in $\Ker \partial_1^K$ can be written as a sum of elements in $\im \partial^K_2$ and a linear combination of $\{z^1_i\}_{i=1,\dots, a_1}$.

\chunk {\bf Products in degree two.}
Let  $\overline{A_2}$ be a $\kk$-subspace of $A_2$ such that
$$A_2=(A_1\cdot A_1)\oplus \overline {A_2}.$$
For each $1\leq \ell\leq a_2-q_{11}$, we consider  $z_\ell^2\in\Ker\partial_2^ K$  such that
\begin{equation}
\label{z2}
 \{[z^2_\ell]\}_{\ell=1,\dots ,a_2-q_{11}} \quad \text{is a basis of}\  \overline{A_2}.
 \end{equation}
In particular, every element in $\Ker \partial_2^K$ can be written as a sum of elements in $\im \partial^K_3$, a linear combination of $\{z^1_i\wedge z^1_j\}_{i,j=1, \dots, a_1}$, and a linear combination of $\{z^2_\ell\}_{\ell=1, \dots, a_2-q_{11}}$.\\

Consider the multiplication map:
\begin{equation}
  \label{def phi}
\phi_1:  A_1\otimes A_1 \to A_2,\quad\text{defined by}\quad
\phi_1([x]\otimes[y])=[x]\wedge[y].
\end{equation}
As $\im \phi_1=A_1\cdot A_1$, we have   $\rank_{\kk}(\Ker\phi_1)=a_1^2-q_{11}$.
For all $1\leq i\leq a_1$ and $1\leq s\leq a_1^2-q_{11}$, we choose ${\widetilde p}^1_{si}\in\Ker \partial_1^K$ such that
\begin{equation}
    \label{def p1}
    \Big\{\sum_{i=1}^{a_1}[z^1_i]\otimes[\widetilde p^1_{si}]\Big\}_{s=1,\dots,a_1^2-q_{11}}\quad\text{is a basis of $\Ker\phi_1$.}
  \end{equation}
For each $s$,  we have $\sum_{i=1}^{a_1}[z^1_i] \wedge [\widetilde p^1_{si}]=0$, so there exists ${\widetilde\pi}^3_s\in K_3$ such that
\begin{equation}
  \label{def pi3}
\partial_3^K({\widetilde\pi}^3_s)=\sum_{i=1}^{a_1}{ z^1_i\wedge{\widetilde p}^1_{si}}.
\end{equation}

\begin{lemma}
  \label{ind p1}
  Let ${\widetilde p}^1_{si}\in \Ker \partial^K_1$ be as in \eqref{def p1}. Then the vectors $\{\big([\widetilde p^1_{s1}],\dots, [\widetilde p^1_{sa_1}]\big)\}_{s=1,\dots,a_1^2-q_{11}}$ in $A_1^{a_1}$ are linearly independent.
\end{lemma}
\begin{proof}
  If $\alpha_s$ is in $R$ such that $\sum_{s=1}^{a_1^2-q_{11}}\big([\widetilde p^1_{s1}],\dots, [\widetilde p^1_{sa_1}]\big)\wedge [\alpha_s]=0$, then $$\sum_{s=1}^{a_1^2-q_{11}}[\widetilde p^1_{si}]\wedge [\alpha_s]=0,\quad\text{ for all } 1\leq i\leq a_1.$$
  By tensoring with $[z^1_i]$ and taking the sum over $i$ we get
  $$
  \sum_{i=1}^{a_1}[z^1_i]\otimes \sum_{s=1}^{a_1^2-q_{11}}[\widetilde p^1_{si}]\wedge [\alpha_s]
  =\sum_{s=1}^{a_1^2-q_{11}}\Big(\sum_{i=1}^{a_1}[z^1_i]\otimes [\widetilde p^1_{si}]\Big)\wedge [\alpha_s]
  =0.
  $$  The desired conclusion now follows from \eqref{def p1}.
\end{proof}

\begin{proposition}
\label{z1z2}
  Let $\{[z^1_i]\}_{i=1,\dots, a_1}$ and $\{[z^2_\ell]\}_{\ell=1,\dots,a_2-q_{11}}$ be as in \eqref{z1} and  \eqref{z2}, respectively.
  If  $\beta_{ij}$ and $\alpha_\ell$ are  in $R$ such that
$$\sum_{i,j=1}^{a_1}[z^1_i]\wedge[z_j^1]\wedge[\beta_{ij}]+ \sum_{\ell=1}^{a_2-q_{11}} [z^2_\ell]\wedge[\alpha_\ell] =0,$$
then the following hold:
\begin{enumerate}[$(a)$]
  \item $[\alpha_\ell]=0$ for all $1\leq \ell \leq a_2-q_{11}$;
\item There exist $\gamma_s$ in $R$ such that  for all $1\leq i\leq a_1$
$$\sum_{j=1}^{a_1}[z_j^1]\wedge[\beta_{ij}]=\sum_{s=1}^{a_1^2-q_{11}}[\widetilde p^1_{si}]\wedge[\gamma_s],$$
where $\widetilde p^1_{si}$ is as in \eqref{def p1}.
\end{enumerate}
\end{proposition}

\begin{proof} (a): The first  sum in the hypothesized equality is in $A_1 \cdot A_1$. The elements $\{ [z^2_\ell] \}_{\ell=1, \dots, a_2-q_{11}}$ form a basis  in $A_2$ that completes a basis of  $A_1 \cdot A_1$, thus  $[\alpha_\ell]=0$ for all $1\leq \ell\leq a_2-q_{11}$.

  (b): By (a) it follows that $0=\sum_{i,j=1}^{a_1}[z^1_i]\wedge[z_j^1]\wedge[\beta_{ij}]=\sum_{i=1}^{a_1}[z^1_i]\wedge \Big(\sum_{j=1}^{a_1}[z_j^1]\wedge [\beta_{ij}]\Big),$ which implies that
$\sum_{i=1}^{a_1}[z^1_i]\otimes \Big(\sum_{j=1}^{a_1}[z_j^1]\wedge [\beta_{ij}]\Big)$ is in $\Ker\phi_1.$
The desired conclusion now follows from \eqref{def p1}.
\end{proof}

\chunk {\bf Products in degree three.}
Let  $\overline{A_3}$ be a $\kk$-subspace of $A_3$ such that
$$A_3=(A_1\cdot A_2)\oplus \overline {A_3}=(A_1\cdot A_1\cdot A_1+A_1\cdot\overline{A_2})\oplus \overline{A_3}.$$
For each $1\leq t \leq a_3-q_{12}$, we consider  $z_t^3\in\Ker\partial_3^ K$  such that
\begin{equation}
\label{z3}
\{[z^3_t]\}_{t=1,\dots ,a_3-q_{12}} \quad \text{is basis of}\ \overline{A_3}.
\end{equation}
In particular, every element in $\Ker\partial_3^K$ can be written as a sum of  elements in $\im\partial^K_4$, a linear combination of   $\{z^1_i\wedge z^1_j\wedge z^1_k\}_{i,j,k=1,\dots, a_1}$, a linear combination of $\{z^1_i\wedge z^2_\ell\}_{\substack{i=1,\dots,a_1\\ \ell=1,\dots,a_2-q_{11}}}$,
and a linear combination of $\{z^3_t\}_{t=1, \dots, a_3-q_{12}}$. \\

Consider the map:
\begin{equation}
  \label{def psi} \psi : A_1\otimes A_1\otimes A_1 \to \big((A_1\cdot A_1)\otimes A_1\big)\oplus \big(A_1\otimes (A_1\cdot A_1)\big),
  \end{equation} defined by
$$\psi([x]\otimes [y]\otimes [z])=\Big([x]\wedge[y]\otimes [z],[x]\otimes[y]\wedge [z]\Big)$$ for all $[x],[y],[z]\in A_1$ and set

\begin{equation}
  \label{def b}
b=\rank_{\kk} (\Coker \psi).
\end{equation}
Remark that  $b=0$ if and only if $\psi$ is surjective. If  $q_{11}=0$, then $A_1\cdot A_1=0$ and hence $b=0$.

\begin{lemma}
  \label{complex} Let $a_i,q_{ij}$ be as in \eqref{a_i} and $b$ be as in \eqref{def b}. The sequence $$A_1\otimes A_1\otimes A_1\xrightarrow{\overline\psi} \big((A_1\cdot A_1)\otimes A_1\big)\oplus \big(A_1\otimes A_2\big)\xrightarrow {\mu} A_3,$$
  where
  \begin{align*}
  \overline\psi([x]\otimes [y]\otimes [z])
  &=\Big([x]\wedge[y]\otimes [z],[x]\otimes[y]\wedge [z]\Big),\, \text{and}
  \\
  \mu(([u]\otimes [x],[y]\otimes[v]))
  &=[u]\wedge[x]-[y]\wedge[v],
  \end{align*}
  for all $[x],[y],[z]\in A_1$, $[u]\in (A_1\cdot A_1)$ and $[v]\in A_2$,
is a complex whose homology has rank
$$a_1a_2-a_1q_{11}-q_{12}+b.$$
\end{lemma}

\begin{proof}  We first show that $\mu\circ \overline\psi =0$.
  \begin{align*}
  \mu(\overline\psi([x]\otimes [y]\otimes [z]))
  =\mu\Big(\Big([x]\wedge[y]\otimes [z],[x]\otimes[y]\wedge [z]\Big)\Big)
  =[x]\wedge[y]\wedge[z]-[x]\wedge [y]\wedge [z]=0.
\end{align*}
Hence, the sequence in the statement of the lemma is a complex. The homology of this complex has
\begin{align*}
  \rank_{\kk}\Big(\frac{\Ker \mu}{\im \overline\psi}\Big)
  &= \rank_{\kk}(\Ker\mu)-\rank_{\kk}(\im\overline\psi)
  \\
  &=\rank_{\kk}\Big(\big((A_1\cdot A_1)\otimes A_1\big)\oplus \big(A_1\otimes A_2\big)\Big)-\rank_{\kk}(\im\mu)-\rank_{\kk}(\im \psi)
  \\
  &=a_1q_{11}+a_1a_2-q_{12}-\rank_{\kk}\Big(\big((A_1\cdot A_1)\otimes A_1\big)\oplus \big(A_1\otimes (A_1\cdot A_1)\big)\Big)+b\\
  &=a_1q_{11}+a_1a_2-q_{12}-2a_1q_{11}+b\\
  &=a_1a_2-a_1q_{11}-q_{12}+b.
\end{align*}
\end{proof}

To describe a basis of the homology of the complex in Lemma \ref{complex}, it is enough to observe that $\im\overline\psi$ is generated by elements of the form  $([z^1_i]\wedge[z^1_j]\otimes [z^1_k],[z^1_i]\otimes[z^1_j]\wedge [z^1_k])$ for $1\leq i,j,k\leq a_1$ and that in ${\Ker \mu}/{\im \overline\psi}$ we have
$ \left[([z^1_i]\wedge[z^1_j]\otimes [z^1_k], 0)\right]= -\left[(0, [z^1_i]\otimes[z^1_j]\wedge [z^1_k])\right]$,  as
  $[([z^1_i]\wedge[z^1_j]\otimes [z^1_k],[z^1_i]\otimes[z^1_j]\wedge [z^1_k])] =0.$
We choose $\widetilde p^2_{ui}\in\Ker\partial_2^K$ such that
\begin{equation}
  \label{def p2}
\Big\{\Big[0, \Big(\sum_{i=1}^{a_1} [z^1_i]\otimes[\widetilde p^2_{ui}]\Big)\Big]\Big\}_{u=1, \dots, a_1a_2-a_1q_{11}-q_{12}+b,} 
\end{equation}
is a basis of the homology of the complex defined in Lemma \ref{complex}.
By definition of $\mu$, for each $u$  we have $\sum_{i=1}^{a_1}[z^1_i]\wedge[{\widetilde p}^2_{ui}]=0$ in $A_3$. Therefore, for each $u$
there exists ${\widetilde\pi}^4_u\in K_4$ such that
\begin{equation}
  \label{def pi4}
\partial_4^K({\widetilde\pi}^4_u)=\sum_{i=1}^{a_1}z^1_i\wedge{\widetilde p}^2_{ui}.
\end{equation}

The following result gives a method for finding $\widetilde p^2_{ui}$ from \eqref{def p2}.

\begin{proposition}
  \label{ker2}
    Consider the map  $\phi_2: A_1\otimes A_2\to A_3$, defined by $\phi_2([x]\otimes[v])=[x]\wedge[v]$, and set
    \begin{align*}
  \CA&= \Big\{\sum_{i=1}^{a_1}[z^1_i]\otimes[\widetilde p^1_{si}]\wedge[z^1_j] \, \big|\,  1\leq  j\leq a_1\, \text{and }\,  1\leq s\leq a_1^2-q_{11}\Big\}\quad\text{and}\\
  \CB& = \Big\{\sum_{i=1}^{a_1}[z^1_i]\otimes[\widetilde p^2_{ui}]\, \big|\, 1\leq u\leq a_1a_2-a_1q_{11}-q_{12}+b\Big\},
\end{align*}
    where $\widetilde p^1_{si}$ and $\widetilde p^2_{ui}$ are as in \eqref{def p1} and \eqref{def p2} respectively. Then, the following hold:

\begin{enumerate}[$(a)$]
  \item The set $\CB$ is  linearly independent and $$\Ker\phi_2=(\Span_\kk\CA)\oplus(\Span_\kk\CB).$$
  \item $\Span_\kk\CA\subseteq A_1\otimes (A_1\cdot A_1)$ and for $b$ defined in \eqref{def b}, $$b=a_1q_{11}-\rank_\kk(\Span_\kk\CA).$$
  \item If $B$ is a $\kk$-subspace of $\Ker\phi_2$ such that  $$\Ker\phi_2=(\Span_\kk\CA)\oplus B,$$ then for every basis $\CB'$ of $B$ the set  $[(0,\CB')]$ is a basis of the homology of the complex in Lemma \ref{complex}.
 \end{enumerate}
\end{proposition}

  \begin{proof} (a):  The linear independence of $\CB$  follows from  \eqref{def p2}. Next, we show that the elements of $\CA$ and $\CB$ generate the kernel of $\phi_2$.  By  definitions of $\widetilde p^1_{si}$ and $\widetilde p^2_{ui}$, it is clear that the listed elements are in the kernel of $\phi_2$.
    Let $\sum_{i=1}^{a_1}[z^1_i]\otimes[p^2_i]\in\Ker\phi_2$, for some $[p^2_i]\in A_2$. Then $[(0,\sum_{i=1}^{a_1}[z^1_i]\otimes[p^2_i])]\in\Ker\mu$, and thus by definition \eqref{def p2} for all $u,i,j$ there exist $\delta_u\in R$  and  $p^1_{ij}\in\Ker\partial_1^K$  such that

  \begin{align}
    \label{ker mu}
  \Big(0,\sum_{i=1}^{a_1}[z^1_i]\otimes [p^2_i]\Big)
  &=\sum_{u=1}^{\tiny\begin{gathered}a_1a_2-a_1q_{11}\vspace{-.1cm}\\-q_{12}+b\end{gathered}}\Big(0,\sum_{i=1}^{a_1}[z^1_i]\otimes [\widetilde p^2_{ui}]\wedge[\delta_u]\Big) +\sum_{i,j=1}^{a_1}\Big([z^1_i]\wedge [p^1_{ij}]\otimes [z_j^1], [z^1_i]\otimes [p^1_{ij}]\wedge[z_j^1]\Big).
  \end{align}
    Comparing the first components of both sides of  \eqref{ker mu}, for each $j$ we have:
    $$0=\sum_{i=1}^{a_1}[z_i^1]\wedge[p^1_{ij}].$$
    By  \eqref{def p1}, there exist $\varepsilon_{js}$ in $R$ such that for all $1\leq i,j\leq a_1$ we can write:
  \begin{equation*}
    [p^1_{ij}]= \sum_{s=1}^{a_1^2-q_{11}}[\widetilde p^1_{si}]\wedge[\varepsilon_{js}].
  \end{equation*}
    Comparing the second components of both sides of \eqref{ker mu}, and using the above expression for $[p^1_{ij}]$, for each $i$  we have:
    \begin{align*}
    [p^2_i]
    &=\sum_{u=1}^{\tiny\begin{gathered}a_1a_2-a_1q_{11}\vspace{-.1cm}\\
    -q_{12}+b\end{gathered}}[\widetilde p^2_{ui}]\wedge[\delta_u]
    +\sum_{j=1}^{a_1}[p^1_{ij}]\wedge [z^1_j]
    =\sum_{u=1}^{\tiny\begin{gathered}a_1a_2-a_1q_{11}\vspace{-.1cm}\\
    -q_{12}+b\end{gathered}}[\widetilde p^2_{ui}]\wedge[\delta_u]
    +\sum_{s=1}^{a_1^2-q_{11}}\sum_{j=1}^{a_1}[\widetilde p^1_{si}]\wedge[z^1_j]\wedge[\varepsilon_{js}].
    \end{align*}
Thus, $\sum_{i=1}^{a_1}[z^1_i]\otimes[p^2_i]\in\Ker\phi_2$ is a linear combination of elements of $\CA$ and $\CB$. 
 Next, we show that $\Span_\kk\CA\cap\Span_\kk\CB=\{0\}$.
  By definition of $\widetilde p^1_{si}$ we have $\sum_{i=1}^{a_1}[z^1_i]\wedge[\widetilde p^1_{si}]=0$ for all $1\leq s \leq a_1^2-q_{11}$. By definition of $\overline\psi$, for all $s$ and $j$ we obtain
  \begin{equation}
    \label{barpsi}
  \overline \psi \Big(\sum_{i=1}^{a_1}[z^1_i]\otimes[\widetilde p^1_{si}]\otimes [z^1_j]\Big)=\Big(0,\sum_{i=1}^{a_1}[z^1_i]\otimes[\widetilde p^1_{si}]\wedge[z^1_j]\Big),
  \end{equation} hence $(0,\CA)\subseteq\im\overline\psi$.
  On the other hand by \eqref{def p2}, $[(0,\CB)]$ is a basis of the homology of the complex in Lemma \ref{complex}. Therefore,  $\Span_\kk\CA$ and $\Span_\kk\CB$ have no nontrivial elements in common, and part (a) holds.

\vspace{0.1in}

  (b): The inclusion follows from the definition of $\CA$. Observe that $\rank_\kk(\Ker\phi_2)=a_1a_2-q_{12}$, and by the linear independence of $\CB$ from part (a), $\rank_\kk(\Span_\kk\CB)=a_1a_2-q_{12}-(a_1q_{11}-b)$. Therefore, by part (a), $\rank_\kk(\Span_\kk\CA)=a_1q_{11}-b.$

\vspace{0.1in}

(c): By parts (a) and (b), every basis $\CB'$ of $B$ has $a_1a_2-a_1q_{11}-q_{12}+b$ elements.
It is clear that $(0,\CB')\subseteq\Ker\mu$, with $\mu$ as in Lemma \ref{complex} and that the set $[(0,\CB')]$ has at most $a_1a_2-a_1q_{11}-q_{12}+b$ elements.
It is enough to show that it is a generating set for the homology of the complex in Lemma \ref{complex} to conclude that $[(0,\CB')]$ has exactly $a_1a_2-a_1q_{11}-q_{12}+b$ elements, hence it forms a basis for the homology.
Since $\CB\subseteq\Ker\phi_2$, every element of $\CB$ can be written as a linear combination of elements of $\CA$ and $\CB'$. In particular, every basis element in $[(0,\CB)]$ can be written as a linear combination of elements in   $[(0,\CA)]$ and $[(0,\CB')]$. However, each element in  $[(0,\CA)]$ is zero in the homology, as remarked in \eqref{barpsi}, thus $[(0,\CB')]$ is a generating set. Part (c) now holds.
\end{proof}

\begin{remark} In practice, one can apply Proposition~\ref{ker2}(b) to compute the value $b$, instead of using definition~\eqref{def b}. Similarly, instead of using definition~\eqref{def p2} for $\widetilde p^2_{ui}$, by Proposition~\ref{ker2}(c), one can find first $\widetilde{p}^1_{si}$ as defined in \eqref{def p1}, and then find a basis of a space complementary to
$$\Span_\kk\CA=\Span_\kk\left\{\sum_{i=1}^{a_1}[z^1_i]\otimes[\widetilde p^1_{si}]\wedge[z^1_j] \right\}_{1\leq  j\leq a_1,  1\leq s\leq a_1^2-q_{11}}$$ 
in the kernel of the multiplication map $\phi_2$, of the form
$$\left\{\sum_{i=1}^{a_1}[z^1_i]\otimes[\widetilde p^2_{ui}]\right\}_{1\leq u\leq a_1a_2-a_1q_{11}-q_{12}+b}.$$ 
\end{remark}

\begin{proposition}
  \label{z1z2z3}
    Let $\{[z^1_i]\}_{i=1,\dots, a_1}$, $\{[z^2_\ell]\}_{\ell=1,\dots,a_2-q_{11}}$, and  $\{[z^3_t]\}_{t=1,\dots,a_3-q_{12}}$ be as in \eqref{z1}, \eqref{z2}, and  \eqref{z3}, respectively.
 If $\gamma_{ijk}, \beta_{i\ell}$, and $\alpha_{t}$ are  in $R$ such that
$$
\displaystyle\sum_{i,j,k=1}^{a_1}[z^1_i]\wedge [z^1_j]\wedge[z^1_k]\wedge[\gamma_{ijk}]
+\sum_{i=1}^{a_1}  \sum_{\ell=1}^{a_2-q_{11}} [z^1_i]\wedge [z^2_\ell] \wedge [\beta_{i\ell}]
+ \sum_{t=1}^{a_3-q_{12}}[z^3_t]\wedge[\alpha_t
]=0,
$$
then the following hold:
\begin{enumerate}[$(a)$]
  \item $[\alpha_t]=0$ for all $1\leq t \leq a_3-q_{12}$;
\item There exist $\delta_{u}$ and $\varepsilon_{js}$ in $R$ such that
\begin{equation*}
\sum_{j,k=1}^{a_1}[z^1_j]\wedge[z^1_k]\wedge[\gamma_{ijk}]
+\sum_{\ell=1}^{a_2-q_{11}} [z^2_\ell] \wedge [\beta_{i\ell}]
 =\sum_{j=1}^{a_1}\sum_{s=1}^{a_1^2-q_{11}}[\widetilde p^1_{si}]\wedge[z^1_j]\wedge[\varepsilon_{js}]
+\hspace{-0.3cm}\sum_{u=1}^{\tiny\begin{gathered}a_1a_2-a_1q_{11}
\vspace{-0.1cm}\\
-q_{12}+b\end{gathered}} \hspace{-0.3cm}[\widetilde p^2_{ui}]\wedge[\delta_u],
\end{equation*}
where $\widetilde p^1_{si}$ and $\widetilde p^2_{ui}$ are as in \eqref{def p1} and \eqref{def p2} respectively.
\end{enumerate}
\end{proposition}
\begin{proof}
(a): The first two sums in the hypothesized equality are in $A_1 \cdot A_2$. As $\{ [z^3_t] \}_{t=1,\dots,a_3-q_{12}}$ is a basis in $A_3$ that completes a basis of $A_1 \cdot A_2$, this implies $[\alpha_t]=0$ for all $1\leq t\leq a_3-q_{12}$.

(b): By (a) it follows that $$\sum_{i=1}^{a_1}[z^1_i]\wedge \Big(\sum_{j,k=1}^{a_1}[z^1_j]\wedge[z^1_k]\wedge[\gamma_{ijk}]
+\sum_{\ell=1}^{a_2-q_{11}} [z^2_\ell] \wedge [\beta_{i\ell}]\Big)=0,$$
hence
$\sum_{i=1}^{a_1}[z^1_i]\otimes\big(\sum_{j,k=1}^{a_1}[z^1_j]\wedge[z^1_k]\wedge[\gamma_{ijk}]
+\sum_{\ell=1}^{a_2-q_{11}} [z^2_\ell] \wedge [\beta_{i\ell}]\big)$ is in $\Ker\phi_2$.
The desired assertion now follows from Proposition \ref{ker2}(a).
\end{proof}

\chunk {\bf Massey products in degree four.}
Massey products occur in degrees four and higher. In degree four, one may obtain only triple Massey products of elements of degree one.
The {\it Massey  product} of a triplet $[x], [y], [z]\in A_1$ satisfying
$$[x]\wedge [y]=0\quad\text{and}\quad [y]\wedge [z]=0,$$
is a subset of $A_4$, defined as follows:
$$\langle [x], [y], [z]\rangle = \{ [\pi^3_{xy}\wedge z+ x\wedge \pi^3_{yz}]\mid \partial^K_3(\pi^3_{xy})=x\wedge y\ \text{and}\ \partial^K_3(\pi^3_{yz})=y \wedge z\},$$
for some $\pi^3_{xy},\pi^3_{yz}\in K_3$.

\begin{remark} Let $[x], [y], [z]\in A_1$ with  $[x]\wedge [y]=0\, \text{and}\, [y]\wedge [z]=0.$
Choose $\pi^3_{xy}$ and $\pi^3_{yz}$ in $K_3$ such that $\partial^K_3(\pi^3_{xy})=x\wedge y\ \text{and}\ \partial^K_3(\pi^3_{yz})=y \wedge z$. Then,  every element of $\langle [x], [y], [z]\rangle$ is of the form
\begin{align*}
  [(\pi^3_{xy}+p^3_{xy})\wedge z+ x\wedge (\pi^3_{yz}+p^3_{yz})]&=[\pi^3_{xy}\wedge z+ x\wedge \pi^3_{yz}]+[p^3_{xy}]\wedge [z]+[x]\wedge [p^3_{yz}],
\end{align*}
  for some $p^3_{xy}$ and $p^3_{yz}\in \Ker\partial_3^K$. Therefore,
$$\langle [x], [y], [z]\rangle=[\pi^3_{xy}\wedge z+ x\wedge \pi^3_{yz}] + (A_3\cdot [z]) +([x]\cdot A_3).$$
The element $\pi^3_{xy}\wedge z+ x\wedge \pi^3_{yz}$ is called  {\it a representative} of the triple Massey product
$\langle [x], [y], [z]\rangle$. Here, ($A_3\cdot [z] + [x]\cdot A_3)$ is called the {\it indeterminacy} of the Massey operation, see e.g., May \cite{M} and \cite[Section 7]{Av74}.
\end{remark}

Let  $\langle A_1, A_1, A_1\rangle$  denote the set of elements in $A_4$ which are  Massey products of triplets in $A_1$.
The next result describes the elements of this set.

\begin{proposition}
\label{massey}
Each element in  $\langle A_1, A_1, A_1\rangle\subseteq A_4$  has a representative
\begin{equation*}
\sum_{s=1}^{a_1^2-q_{11}}\widetilde\pi^3_s\wedge  p^1_s
\quad
\text{such that}
\quad
\sum_{s=1}^{a_1^2-q_{11}}[\widetilde p^1_{si}\wedge p^1_s]=0,\ \text{for all}\ i=1\dots,a_1,
\end{equation*}
 where $p^1_s\in\Ker\partial_1^K$, and $\widetilde p^1_{si}$ and  $\widetilde \pi^3_{s}$ are  defined in \eqref{def p1}  and \eqref{def pi3}
respectively.
\end{proposition}
\begin{proof}
  Let $[x], [y],[z]\in A_1$ such that
  $[x]\wedge [y]=0\  \text{and}\  [y]\wedge [z]=0.$
We write  $[y]=\displaystyle\sum_{i=1}^{a_1}[z^1_i]\wedge [\alpha_i]$ for some $\alpha_i\in R$.
Since $[x]\wedge [y]= \displaystyle\sum_{i=1}^{a_1}[z^1_i]\wedge [x]\wedge[-\alpha_i]=0$, we have
$\displaystyle\sum_{i=1}^{a_1}[z^1_i] \otimes[x]\wedge[-\alpha_i]\in\Ker\phi_1$.
Hence, by \eqref{def p1} there exist $\beta_{s}\in R$ such that
$[x]\wedge [-\alpha_i]=\sum_{s=1}^{a_1^2-q_{11}}[\widetilde p^1_{si}]\wedge[\beta_s]\ \text{for all}\ i,
$
so
\begin{align*}
[x]\wedge[y]
&=\sum_{s=1}^{a_1^2-q_{11}}\Big(\sum_{i=1}^{a_1}[z^1_i\wedge \widetilde p^1_{si}]\Big)\wedge[\beta_s]
=\sum_{s=1}^{a_1^2-q_{11}}[\partial_3^K(\widetilde\pi^3_s)]\wedge [\beta_s]=\left[\partial_3^K\Big( \sum_{s=1}^{a_1^2-q_{11}}\widetilde\pi^3_s\wedge\beta_s\Big)\right].
\end{align*}
Therefore, we may choose
$\pi^3_{xy}=  \sum_{s=1}^{a_1^2-q_{11}} \widetilde\pi^3_s\wedge\beta_s$.
Similarly, using the equality $[y]\wedge [z]=0$,  we may choose
$\pi^3_{yz}= \displaystyle \sum_{s=1}^{a_1^2-q_{11}}\widetilde\pi^3_s\wedge\gamma_s$, for some $\gamma_s\in R$, where
$
[z]\wedge [\alpha_i]=\sum_{s=1}^{a_1^2-q_{11}}[\widetilde p^1_{si}]\wedge[\gamma_s]\ \text{for all}\ i.
$

It follows that any element in
$\langle A_1, A_1, A_1\rangle$  has a representative
\begin{equation*}
\pi^3_{xy}\wedge z+ x\wedge \pi^3_{yz}
=\sum_{s=1}^{a_1^2-q_{11}}\widetilde\pi^3_s\wedge (z\wedge \beta_s-x\wedge\gamma_s).
\end{equation*}
If we set  $p^1_s=z\wedge \beta_s-x\wedge\gamma_s$, then for each $1\leq i\leq a_1$ we have
\begin{align*}
  \sum_{s=1}^{a_1^2-q_{11}}[\widetilde p^1_{si}\wedge p^1_s]
  &=\sum_{s=1}^{a_1^2-q_{11}}[\widetilde p^1_{si}]\wedge [z\wedge \beta_s-x\wedge\gamma_s]=\Big(\sum_{s=1}^{a_1^2-q_{11}}[\widetilde p^1_{si}\wedge\beta_s] \Big)\wedge[z] + [x]\wedge \Big(\sum_{s=1}^{a_1^2-q_{11}}[\widetilde p^1_{si}\wedge\gamma_s]\Big)\\
  &=-[x]\wedge [z]\wedge[\alpha_i] +[x]\wedge[z]\wedge[\alpha_i]=0.
  \end{align*}
\end{proof}

\chunk {\bf Products in degree four.}
Let $\overline{A_4}$ be a $\kk$-subspace of $A_4$  such that
$$A_4= (A_1\cdot A_3+A_2\cdot A_2+\text{Span}_{\kk}\langle A_1,A_1,A_1\rangle)\oplus \overline{A_4}$$
and set
\begin{equation}
\label{a}
a=\rank_{\kk} \Big( A_1\cdot A_3+A_2\cdot A_2+\text{Span}_{\kk}\langle A_1,A_1,A_1\rangle\Big).
\end{equation}
For some choice of $z^4_r\in\Ker\partial_4^K$, let \begin{equation}
\label{z4}
\{[z^4_r]\}_{r=1, \dots, a_4-a} \quad \text{be a basis of} \ \overline{A_4}.
\end{equation}
In particular, every element in $\Ker\partial_4^K$ can be written as a sum of  elements in $\im\partial^K_5$, a linear combination of   $\{z^1_i \wedge p^3_i\}_{i=1,\dots, a_1}$, a linear combination of $\{z^2_\ell\wedge p^2_\ell\}_{\ell=1,\dots,a_2-q_{11}}$, a linear combination of
$\{\widetilde \pi^3_s \wedge p^1_s \}_{s=1,\dots,a_1^2-q_{11}}$,
and a linear combination of $\{z^4_r\}_{r=1,\dots,a_4-a}$, where $p^3_i\in\Ker\partial^K_3,  p^2_\ell\in\Ker\partial^K_2$ and $p^1_s\in\Ker\partial^K_1$, such that $p^1_s$ is as in Proposition~\ref{massey} and $\widetilde \pi^3_s$ as in \eqref{def pi3}.

\begin{summary} 
\label{summary} 
We summarize here all notations, introduced in this section, to be referred to throughout the rest of the paper:
\begin{align*}
   A&=\HH{}{K^R}=A_0 \oplus A_1 \oplus A_2 \oplus \cdots \\
 A_2&= (A_1 \cdot A_1) \oplus \overline{A_2}\\
 A_3&= (A_1 \cdot A_2) \oplus \overline{A_3}\\
 A_4&= \big(A_1 \cdot A_3 + A_2 \cdot A_2 + \Span_\kk(\langle A_1,A_1,A_1\rangle)\big) \oplus \overline{A_4}\\
&\\
     a_i&=\rank_\kk A_i,\ \text{for}\ 1\leq i\leq 4,\   \eqref{a_i}\\
     q_{ij}&=\rank_\kk A_i \cdot A_j,\ \text{for}\ 1\leq i\leq j\leq 4,\  \eqref{a_i}\\
    a&=\rank_\kk \big( A_1\cdot A_3+A_2\cdot A_2+\text{Span}_{\kk}\langle A_1,A_1,A_1\rangle\big),\  \eqref{a}\\
    &\\
    \psi\colon A_1\otimes A_1\otimes A_1&\to\big((A_1\cdot A_1)\otimes A_1\big)\oplus \big(A_1\otimes (A_1\cdot A_1)\big),\  \eqref{def psi}\\
    \psi([x]\otimes [y]\otimes [z])&=\Big([x]\wedge[y]\otimes [z],[x]\otimes[y]\wedge [z]\Big)\ \text{for all}\  [x],[y],[z]\in A_1\\
     b&=\rank_\kk(\Coker \psi), \eqref{def b}, \text{ see also Proposition~\ref{ker2}(b)}\\
&\\
   \{[z^1_i]   \}&_{i=1,\dots,a_1}\quad\text{is a basis for}\  A_1, \eqref{z1}\\
   \{[z^2_\ell]\}&_{\ell=1,\dots,a_2-q_{11}}\quad\text{is a basis for}\  \overline{A_2}, \eqref{z2}\\
   \{[z^3_t]   \}&_{t=1,\dots,a_3-q_{12}}\quad \text{is a basis for}\ \overline{A_3}, \eqref{z3}\\
   \{[z^4_r]   \}&_{r=1,\dots,a_4-a}\quad \text{is a basis for}\  \overline{A_4}, \eqref{z4}\\
&\\
  \phi_1\colon A_1\otimes A_1&\to A_2,\ \phi_1([x]\otimes[y])=[x]\wedge[y],\ \text{for all}\ [x], [y]\in A_1, \eqref{def phi} \\
\Big\{\sum_{i=1}^{a_1}[z^1_i]\otimes[\widetilde p^1_{si}]\Big\}&_{s=1,\dots,a_1^2-q_{11}}\quad\text{is a basis of $\Ker\phi_1$, \eqref{def p1}} \\
\partial_3^K({\widetilde\pi}^3_s)&=\sum_{i=1}^{a_1}{ z^1_i\wedge{\widetilde p}^1_{si}}, \ \eqref{def pi3}\\
&\\
  \phi_2\colon A_1\otimes A_2&\to A_3,\ 
  \phi_2([x]\otimes[y])=[x]\wedge[y],\ \text{for all}\ [x]\in A_1, [y]\in A_2, \text{Proposition~\ref{ker2}} \\
  \Big\{\sum_{i=1}^{a_1}[z^1_i]\otimes[\widetilde p^2_{ui}]\Big\}&_{1\leq u\leq a_1a_2-a_1q_{11}-q_{12}+b} \quad\text{is a basis of $B$, where} \\
  \Ker\phi_2&= B\oplus \Span_\kk\Big\{\sum_{i=1}^{a_1}[z^1_i]\otimes[\widetilde p^1_{si}]\wedge[z^1_j] \Big\}_{1\leq  j\leq a_1,  1\leq s\leq a_1^2-q_{11}}\\
  \partial_4^K({\widetilde\pi}^4_u)&=\sum_{i=1}^{a_1}z^1_i\wedge{\widetilde p}^2_{ui},\ \eqref{def pi4}.
\end{align*}
\end{summary}

\begin{lemma}
\label{lem: Koszul}
  All the elements  $z^1_i, z^2_\ell, z^3_t, z^4_r, \widetilde p^1_{si}, \widetilde\pi^3_s, \widetilde p^2_{ui}$, and $\widetilde\pi^4_u$, as in Summary~\ref{summary}, are in $\fm K$.
\end{lemma}

\begin{proof} From the inclusions $\Ker\partial_j^K\subseteq \fm K_j$ for all $j\geq 0$, we obtain that $z^1_i, z^2_\ell, z^3_t, z^4_r,  \widetilde p^1_{si}$, and  $\widetilde p^2_{ui}$ are in
$\fm K$.
The assertions for the preimages $\widetilde\pi^3_s$ and  $\widetilde\pi^4_u$ follows from the inclusions
$(\partial^K_j)^{-1}(\fm^2K_{j-1})\subseteq \fm K_j$ for $j=3$ and $4$ respectively.
\end{proof}


\section{Truncated minimal free resolution of the residue field}
\label{sec resolution}

In this section, we construct the beginning of the minimal free resolution of the residue field $\kk$ over the local ring $(R,\fm,\kk)$,
by using the Koszul complex $K$ of $R$ and the graded-commutative structure of the algebra $A=\HH{}{K}$ described in Section~\ref{sec multiplication}.

\begin{construction}
  \label{construction}
We consider the following sequence of free $R$-modules:
\begin{equation*}
{F} :\qquad F_5\xrightarrow{\partial_5^F}F_4\xrightarrow{\partial_4^F}F_3\xrightarrow{\partial_3^F}F_2\xrightarrow{\partial_2^F}F_1\xrightarrow{\partial_1^F}F_0,
\end{equation*}
where
\begin{align*}
F_0:=&K_0\\
F_1:=&K_1\\
F_2:=&K_2\oplus K_0^{a_1}\\
F_3:=&K_3\oplus K_1^{a_1}\oplus K_0^{a_2-q_{11}}\\
F_4:=&K_4\oplus K_2^{a_1}\oplus K_1^{a_2-q_{11}} \oplus  K_0^{a_3-q_{12}}\oplus K_0^{a_1^2-q_{11}}\\
F_5:=&K_5\oplus K_3^{a_1}\oplus K_2^{a_2-q_{11}} \oplus K_1^{a_3-q_{12}}\oplus  K_0^{a_4-a}\oplus K_1^{a_1^2-q_{11}}\oplus K_0^{a_1a_2-a_1q_{11}-q_{12}+b} \oplus K_0^{a_1a_2- a_1q_{11}}.
\end{align*}
Using the elements described in Section~\ref{sec multiplication}, the  differential maps of $F$ are defined as follows.
\begin{equation}
\label{d1}
\partial_1^F: K_1\to K_0,\quad\mbox{is given by}\quad \partial_1^F:= \partial_1^{K}.
\end{equation}

\begin{equation}
\label{d2}
\partial_2^F: K_2\oplus K_0^{a_1}\to K_1,\quad\mbox{is given by}\quad \partial_2^F:=
\begin{pmatrix}\partial_2^K&z^1\wedge\end{pmatrix},
\end{equation}
that is
\begin{equation*}
\partial_2^F
\begin{pmatrix}
\pi^2\\[0.1cm]
(\alpha_i)_{i=1,\dots, a_1}
\end{pmatrix}
:=\partial_2^{K}(\pi^2)+\sum_{i=1}^{a_1}z_i^1\wedge \alpha_i.
\end{equation*}

\begin{equation}
  \label{d3}
\partial_3^F: K_3\oplus K_1^{a_1}\oplus K_0^{a_2-q_{11}}\to  K_2\oplus K_0^{a_1},\quad\mbox{is given by}\quad
\partial_3^F: =
\begin{pmatrix}\partial_3^K&z^1\wedge & - z^2\wedge \\[0.1cm]
  0&(\partial_1^K)^{a_1}&0
\end{pmatrix},
\end{equation}
that is
\begin{equation*}
\partial_3^F
\begin{pmatrix}
\pi^3\\[0.1cm]
 (\pi^1_i)_{i=1,\dots,a_1}\\[0.1cm]
(\alpha_\ell)_{\ell=1,\dots,a_2-q_{11}}
\end{pmatrix}
:=\begin{pmatrix}
  \displaystyle
\partial^K_3(\pi^3)+\sum_{i=1}^{a_1}z^1_i\wedge \pi^1_i - \sum_{\ell=1}^{a_2-q_{11}} z^2_\ell\wedge \alpha_\ell
\\[0.5cm]
  \displaystyle
(\partial^K_1(\pi^1_i))_{i=1,\dots,{a_1}}
\end{pmatrix}.
\end{equation*}

\begin{equation}
\label{d4}
\partial_4^F: K_4\oplus K_2^{a_1}\oplus  K_1^{a_2-q_{11}}\oplus K_0^{a_3-q_{12}}\oplus K_0^{a_1^2-q_{11}} \to  K_3\oplus K_1^{a_1}\oplus K_0^{a_2-q_{11}},\
\text{is given by}
\end{equation}
$$
\partial_4^F:=
\begin{pmatrix}\partial_4^K&z^1\wedge & z^2\wedge&z^3\wedge&-\widetilde\pi^3\wedge \\ 0&(\partial_2^K)^{a_1}&0&0&(\widetilde p^1\wedge)^{a_1}\\
0&0&(\partial_1^K)^{a_2-q_{11}}&0&0
\end{pmatrix},
$$
that is
$$
{\small
\partial_4^F
\begin{pmatrix}
\pi^4 \\[0.1cm]
(\pi^2_i)_{i=1,\dots,a_1}\\[0.1cm]
 (\pi^1_\ell)_{\ell=1,\dots, a_2-q_{11}}\\[0.1cm]
(\alpha_t)_{t=1,\dots,a_3-q_{12}}\\[0.1cm]
 (\beta_s)_{s=1,\dots,a_1^2-q_{11}}
\end{pmatrix}
:=
\begin{pmatrix}
  \displaystyle
\partial^K_4(\pi^4)
+\sum_{i=1}^{a_1}z^1_i\wedge \pi^2_i
+\sum_{\ell=1}^{a_2-q_{11}}z^2_\ell\wedge \pi^1_\ell
+\sum_{t=1}^{a_3-q_{12}}z^3_t\wedge\alpha_t
-\sum_{s=1}^{a_1^2-q_{11}}{\widetilde\pi}^3_s\wedge\beta_s\\[0.5cm]
  \displaystyle
(\partial_2^K(\pi^2_i)+\sum_{s=1}^{a_1^2-q_{11}}{\widetilde p}^1_{si}\wedge \beta_s)_{i=1,\dots,a_1}\\[0.5cm]
  \displaystyle
(\partial^K_1(\pi^1_\ell))_{\ell=1,\dots,a_2-q_{11}}
\end{pmatrix}.}
$$
\begin{equation}
  \label{d5}
  \begin{aligned}
\partial^F_5:
K_5\oplus K_3^{a_1}\oplus K_2^{a_2-q_{11} }\oplus K_1^{a_3-q_{12}}\oplus K_0^{a_4-a}\oplus K_1^{a_1^2-q_{11}}\oplus K_0^{a_1a_2 - a_1q_{11}-q_{12}+b} \oplus K_0^{a_1a_2-a_1q_{11}}\\
\to K_4\oplus K_2^{a_1}\oplus K_1^{a_2-q_{11}}\oplus K_0^{a_3-q_{12}}\oplus   K_0^{a_1^2-q_{11}}
\end{aligned}
\end{equation}
 is given  by
 \begin{equation*}
 \partial_5^F:=
   \begin{pmatrix}
     \partial_5^K&z^1\wedge & - z^2\wedge&z^3\wedge&z^4\wedge&-\widetilde\pi^3\wedge&-\widetilde\pi^4\wedge&0 \\
     0&(\partial_3^K)^{a_1}&0&0&0&(\widetilde p^1\wedge)^{a_1}&(\widetilde p^2\wedge)^{a_1}&-(z^2\wedge)^{a_1}\\
   0&0&(\partial_2^K)^{a_2-q_{11}}&0&0&0&0&(z^1\wedge)^{a_2-q_{11}}\\
   0&0&0&(\partial_1^K)^{a_3-q_{12}}&0&0&0&0\\
   0&0&0&0&(\partial_1^K)^{a_1^2-q_{11}}&0&0&0\\
 \end{pmatrix},
\end{equation*}
that is
{\small
\begin{align*}
\partial_5^F
&\begin{pmatrix}
\pi^5 \\[0.1cm]
(\pi^3_i)_{i=1,\dots,a_1}\\[0.1cm]
(\pi^2_\ell)_{\ell=1,\dots, a_2-q_{11}}\\[0.1cm]
 (\pi^1_t)_{t=1,\dots,a_3-q_{12}}\\[0.1cm]
(\alpha_r)_{r=1,\dots, a_4-a}\\[0.1cm]
(\pi'^1_s)_{s=1,\dots,a_1^2-q_{11}}\\[0.1cm]
(\beta_u)_{u=1,\dots, a_1a_2-a_1q_{11}-q_{12}+b}\\[0.1cm]
(\gamma_{\ell i})_{\ell=1,\dots, a_2-q_{11};\ i=1,\dots, a_1}
\end{pmatrix}
:=
\\
&\begin{pmatrix}\displaystyle
\partial^K_5(\pi^5)
+\sum_{i=1}^{a_1}z^1_i\wedge\pi^3_i
- \hspace{-.3cm}\sum_{\ell=1}^{a_2-q_{11}}z^2_\ell\wedge  \pi^2_\ell
+ \sum_{t=1}^{a_3-q_{12}}z^3_t\wedge  \pi^1_t
+\sum_{r=1}^{a_4-a} z^4_r\wedge \alpha_r
- \hspace{-.3cm}\sum_{s=1}^{a_1^2-q_{11}}\widetilde \pi^3_s\wedge  \pi'^1_s
- \hspace{-.5cm}\sum_{u=1}^{\tiny\begin{gathered}a_1a_2-a_1q_{11}\\-q_{12}+b\end{gathered}}\hspace{-.3cm}\widetilde\pi^4_u\wedge\beta_u\\
\\
\displaystyle
\Big(
\partial_3^K(\pi^3_i)
+\sum_{s=1}^{a_1^2-q_{11}}{\widetilde p}^1_{si}\wedge\pi'^1_s
+\sum _{u=1}^{\tiny\begin{gathered}a_1a_2-a_1q_{11}\\-q_{12}+b\end{gathered}}\widetilde p^2_{ui}\wedge \beta_u
-\sum_{\ell=1}^{a_2-q_{11}}z^2_\ell\wedge\gamma_{\ell i}
\Big)_{i=1,\dots,a_1}\\[0.5cm]
\displaystyle
\Big(\partial_2^K(\pi^2_\ell)
+\sum_{i=1}^{a_1}z^1_i\wedge  \gamma_{\ell i} \Big)_{\ell=1,\dots, a_2-q_{11}}\\[0.5cm]
\displaystyle
\Big(\partial^K_1(\pi^1_t)\Big)_{t=1,\dots,a_3-q_{12}}\\[0.5cm]
\displaystyle
\Big(\partial_1^K(\pi'^1_s)\Big)_{s=1,\dots,a_1^2-q_{11}}
\end{pmatrix}.
\end{align*}
}
\end{construction}

\vspace{0.12in}

\begin{theorem}
  \label{main}
Let $(R,\fm,\kk)$ be a local ring.
The sequence $F$ constructed in \ref{construction} is  a truncated minimal free resolution of $\kk$  over  $R$, up to homological degree five.
\end{theorem}
\begin{proof} The minimality follows from Lemma~\ref{lem: Koszul}.
We show exactness at each degree by using  the Koszul relations in the complex $K$, and the basis elements and maps  defined in Section \ref{sec multiplication}.

\vspace{0.1in}
\noindent
{\bf Exactness at degree one.}
$\im\partial_2^F\subseteq \Ker\partial_1^F$:

\begin{align*}
  \partial_1^F \circ \partial_2^F
\begin{pmatrix}
\pi^2\\[0.1cm]
(\alpha_i)_{i=1,\dots, a_1}
\end{pmatrix}
=\partial_1^K\begin{pmatrix}\partial_2^{K}(\pi^2)+\sum_{i=1}^{a_1}z_i^1\wedge\alpha_i\end{pmatrix}
=0.
\end{align*}

$\im\partial_2^F\supseteq \Ker\partial_1^F$:   If  $\pi^1\in \Ker\partial_1^F$, then there exist $\pi^2\in K_2$  and $\alpha_i\in R$ such that
$$
\pi^1=\partial_2^K(\pi^2)+\sum_{i=1}^{a_1}z_i^1\wedge \alpha_i ,
\quad\text{therefore},\quad \pi^1=\partial_2^F\begin{pmatrix}
\pi^2\\[0.1cm]
(\alpha_i)_{i=1,\dots, a_1}
\end{pmatrix}.$$

\noindent
{\bf Exactness at degree two.}
$\im\partial_3^F\subseteq \Ker\partial_2^F$:

\begin{align*}
\partial_2^F\circ\partial_3^F
\begin{pmatrix}
  \pi^3\\[0.1cm]
  (\pi^1_i)_{i=1,\dots,a_1}\\[0.1cm]
  (\alpha_\ell)_{\ell=1,\dots,a_2-q_{11}}
\end{pmatrix}
&=\partial_2^F
\begin{pmatrix}
\partial^K_3(\pi^3)+\sum_{i=1}^{a_1}z^1_i\wedge\pi^1_i - \sum_{\ell=1}^{a_2-q_{11}} z^2_\ell\wedge \alpha_\ell\\[0.1cm]
(\partial^K_1(\pi^1_i))_{i=1,\dots,{a_1}}
\end{pmatrix}
\\
&=\partial_2^K
\Big(
\partial^K_3(\pi^3)
+\sum_{i=1}^{a_1}z^1_i\wedge\pi^1_i
- \sum_{\ell=1}^{a_2-q_{11}} z^2_\ell\wedge\alpha_\ell
\Big)+\sum_{i=1}^{a_1}z^1_i\wedge\partial^K_1(\pi^1_i)
\\
&=
-\sum_{i=1}^{a_1} z^1_i\wedge \partial_1^K(\pi^1_i)
+\sum_{i=1}^{a_1}z^1_i\wedge \partial^K_1(\pi^1_i) =0.
\end{align*}

$\im\partial_3^F\supseteq \Ker\partial_2^F$:  If
$\begin{pmatrix}
\pi^2\\[0.1cm]
(\alpha_i)_{i=1,\dots, a_1}
\end{pmatrix}\in\Ker\partial_2^F$, then
$\partial_2^{K}(\pi^2)+\sum_{i=1}^{a_1}z_i^1\wedge \alpha_i=0$  in $K_1.$
In particular, we have $\sum_{i=1}^{a_1}[z_i^1]\wedge[\alpha_i]=0$ in $A_1$. Since  $\{[z^1_i]\}_{i=1,\dots, a_1}$ is a basis of $A_1$, we get $\alpha_i\in \fm$.
It follows that for each $i$ there  exists $\pi^1_i\in K_1$ such that
\begin{equation}
\label{eq22}
\alpha_i= \partial_1^K(\pi^1_i).
\end{equation}
Hence  we have
$$0=\partial_2^{K}(\pi^2)+\sum_{i=1}^{a_1} z^1_i\wedge\partial_1^K(\pi^1_i)=\partial_2^{K} \Big(\pi^2-\sum_{i=1}^{a_1}z^1_i\wedge \pi^1_i\Big).$$
In particular,  there exist
$\pi^3\in K_3$, $p^1_i\in\Ker\partial_1^K$, and $\beta_\ell\in R$  such that
$$\pi^2-\sum_{i=1}^{a_1}z^1_i\wedge \pi^1_i=\partial_3^K(\pi^3)+\sum_{i=1}^{a_1} z_i^1\wedge p^1_i+\sum_{\ell=1}^{a_2-q_{11}}z^2_\ell\wedge \beta_\ell,\quad\text{which implies}$$
\begin{equation}
\label{eq23}
 \pi^2=\partial_3^K(\pi^3)+\sum_{i=1}^{a_1}z^1_i\wedge(\pi^1_i+p^1_i)+\sum_{\ell=1}^{a_2-q_{11}} z^2_\ell\wedge \beta_\ell.
\end{equation}
Combining (\ref{eq22}) and (\ref{eq23}) we get
$$ \Ker \partial^F_2 \ni
\begin{pmatrix}
\pi^2\\[0.1cm]
(\alpha_i)_{i=1,\dots, a_1}
\end{pmatrix}
=
\partial_3^F
\begin{pmatrix}
\pi^3\\[0.1cm]
 (\pi^1_i+p^1_i)_{i=1,\dots,a_1}\\[0.1cm]
- (\beta_\ell)_{\ell=1,\dots,a_2-q_{11}}
\end{pmatrix}.
$$

\noindent
{\bf Exactness at degree three.}
$\im\partial_4^F\subseteq \Ker\partial_3^F$:  We show that both components of the element

\noindent
$\partial_3^F \circ \partial_4^F
\begin{pmatrix}
\pi^4 \\[0.1cm]
(\pi^2_i)_{i=1, \dots, a_1}\\[0.1cm]
 (\pi^1_\ell)_{\ell=1, \dots, a_2-q_{11}}\\[0.1cm]
 (\alpha_t)_{t=1, \dots, a_3-q_{12}}\\[0.1cm]
 (\beta_s)_{s=1, \dots, a_1^2-q_{11}}
\end{pmatrix}
$
are zero. The first component is:
\begin{align*}
&\partial_3^K\Big(
\partial^K_4(\pi^4)
+\sum_{i=1}^{a_1}z^1_i\wedge \pi^2_i
+\sum_{\ell=1}^{a_2-q_{11}}  z^2_\ell\wedge\pi^1_\ell
+\sum_{t=1}^{a_3-q_{12}}z^3_t\wedge \alpha_t
- \sum_{s=1}^{a_1^2-q_{11}}{\widetilde\pi}^3_s\wedge\beta_s
\Big)\\
& \qquad
+\sum_{i=1}^{a_1}z^1_i\wedge \partial_2^K(\pi^2_i)
+ \sum_{i=1}^{a_1} z^1_i\wedge \Big(\sum_{s=1}^{a_1^2-q_{11}}{\widetilde p}^1_{si}\wedge\beta_s \Big)
 - \sum_{\ell=1}^{a_2-q_{11}}z^2_\ell\wedge\partial^K_1(\pi^1_\ell)\\
=&
-\sum_{i=1}^{a_1} z^1_i\wedge \partial_2^K(\pi^2_i)
+\sum_{\ell=1}^{a_2-q_{11}}z^2_\ell\wedge \partial^K_1(\pi^1_\ell)
- \sum_{s=1}^{a_1^2-q_{11}}\partial^K_3({\widetilde\pi}^3_s)\wedge\beta_s
\\
& \qquad +\sum_{i=1}^{a_1} z^1_i\wedge \partial_2^K(\pi^2_i)
+ \sum_{s=1}^{a_1^2-q_{11}} \Big(\sum_{i=1}^{a_1}z^1_i\wedge{\widetilde p}^1_{si}\Big)\wedge\beta_s
- \sum_{\ell=1}^{a_2-q_{11}}z^2_\ell\wedge\partial^K_1(\pi^1_\ell)
=0.
\end{align*}
The last equality follows from the definition of ${\widetilde\pi}^3_s$ in \eqref{def pi3}.
As ${\widetilde p}^1_{si}\in\Ker\partial^K_1$ for all $i$ and $s$, we have $\partial_1^K\Big(\partial_2^K(\pi^2_i)+\sum_{s=1}^{a_1^2-q_{11}}{\widetilde p}^1_{si}\wedge\beta_s\Big)
=0,$  for each $1\leq i \leq a_1$,  so the second component is zero.\\

$\im\partial_4^F\supseteq \Ker\partial_3^F$:
If
$\begin{pmatrix}
\pi^3\\[0.1cm]
 (\pi^1_i)_{i=1,\dots,a_1}\\[0.1cm]
(\alpha_\ell)_{\ell=1,\dots,a_2-q_{11}}
\end{pmatrix}\in\Ker\partial_3^F$,
then
\begin{align}
\label{eq31}
\partial^K_3(\pi^3)+\sum_{i=1}^{a_1}z^1_i\wedge\pi^1_i - \sum_{\ell=1}^{a_2-q_{11}} z^2_\ell\wedge \alpha_\ell &=0 \quad \mbox{in}\ K_2, \quad\text{and}\\
\partial^K_1(\pi^1_i)&=0 \quad \mbox{ for all}\  1\leq i\leq {a_1}.
\end{align}
It follows that for each $i$,  there exist $\pi^2_i\in K_2$ and $\beta_{ij}\in R$ such that
\begin{equation}
\label{pi1i}
\pi^1_i=\partial_2^K(\pi^2_i)+\sum_{j=1}^{a_1}z_j^1\wedge\beta_{ij}.
\end{equation}
Thus (\ref{eq31}) becomes:
\begin{equation}
\label{eq34}
\partial^K_3
\Big(\pi^3-\sum_{i=1}^{a_1}z^1_i\wedge\pi^2_i \Big)
+\sum_{i=1}^{a_1}z^1_i\wedge \Big(\sum_{j=1}^{a_1}z_j^1\wedge\beta_{ij}\Big)
- \sum_{\ell=1}^{a_2-q_{11}} z^2_\ell\wedge\alpha_\ell =0.
\end{equation}
In $A_2$ we obtain
$
  \sum_{i,j=1}^{a_1}[z^1_i]\wedge [z_j^1]\wedge [\beta_{ij}] - \sum_{\ell=1}^{a_2-q_{11}} [z^2_\ell]\wedge[\alpha_\ell] =0.
$ By Proposition \ref{z1z2}, we get $[\alpha_\ell]=0$ for all $\ell$ and
$\sum_{j=1}^{a_1} [z_j^1]\wedge [\beta_{ij}]=\sum_{s=1}^{a_1^2-q_{11}}[{\widetilde p}^1_{si}]\wedge[\beta'_s]$ for some $\beta'_s$ in $R$.
In particular, there exist ${\pi'_\ell}^1\in K_1$  and $\pi'^2_i\in K_2$ such that
\begin{align}
\label{alphal}
 \alpha_\ell&=\partial_1^K({\pi'_\ell}^1), \text{ and }
 \\
\sum_{j=1}^{a_1} z_j^1\wedge \beta_{ij}&=\partial^K_2(\pi'^2_i)+\sum_{s=1}^{a_1^2-q_{11}}{\widetilde p}^1_{si}\wedge\beta'_s\nonumber.
\end{align}
The equation \eqref{pi1i} becomes:
\begin{equation}
  \label{pi1i'}
  \pi^1_i=\partial_2^K(\pi^2_i+\pi'^2_i)+\sum_{s=1}^{a_1^2-q_{11}}{\widetilde p}^1_{si}\wedge\beta'_s,
\end{equation}
and equation \eqref{eq34} now becomes:
\begin{align*}
\label{eq35}
0&=\partial^K_3
\Big(\pi^3-\sum_{i=1}^{a_1}z^1_i\wedge\pi^2_i \Big)
+\sum_{i=1}^{a_1}z^1_i\wedge \partial^K_2(\pi'^2_i)
+\sum_{i=1}^{a_1}\sum_{s=1}^{a_1^2-q_{11}}z^1_i\wedge {\widetilde p}^1_{si}\wedge\beta'_s
-  \sum_{\ell=1}^{a_2-q_{11}} z^2_\ell\wedge\partial_1^K({\pi'_\ell}^1)
\\
&
=\partial^K_3
\Big(
\pi^3
-\sum_{i=1}^{a_1}z^1_i\wedge (\pi^2_i+\pi'^2_i)
- \sum_{\ell=1}^{a_2-q_{11}} z^2_\ell\wedge {\pi'_\ell}^1
+ \sum_{s=1}^{a_1^2-q_{11}} {\widetilde\pi}^3_{s}\wedge\beta'_s
\Big),
\end{align*}
where the second equality uses the definition of $\widetilde\pi^3_s$ from \eqref{def pi3}.
Hence, there exist  $\pi^4\in K_4$, $p^2_i\in \Ker \partial_2^K$, $p^1_\ell\in\Ker\partial_1^K$, and $\alpha'_t\in R$ such that
\begin{align*}
&\pi^3
-\sum_{i=1}^{a_1}z^1_i\wedge (\pi^2_i+\pi'^2_i)
- \sum_{\ell=1}^{a_2-q_{11}} z^2_\ell\wedge {\pi'}^1_\ell
+ \sum_{s=1}^{a_1^2-q_{11}} {\widetilde\pi}^3_{s}\wedge\beta'_s
\\
&=\partial_4^K(\pi^4)
+ \sum_{i=1}^{a_1}z^1_i \wedge p^2_i
+\sum_{\ell=1}^{a_2-q_{11}}z^2_\ell\wedge p^1_\ell
+\sum_{t=1}^{a_3-q_{12}}z_t^3\wedge\alpha'_t.
\end{align*}
Thus,
\begin{equation}
\label{pi3}
\pi^3=
\partial_4^K(\pi^4)
+\sum_{i=1}^{a_1}z^1_i\wedge (\pi^2_i+\pi'^2_i+p^2_i)
+\sum_{\ell=1}^{a_2-q_{11}} z^2_\ell\wedge ({\pi'_\ell}^1+p^1_\ell)
+\sum_{t=1}^{a_3-q_{12}}z^3_t\wedge\alpha'_t
-\sum_{s=1}^{a_1^2-q_{11}}{\widetilde\pi}^3_{s}\wedge\beta'_s.
\end{equation}
The equations  \eqref{alphal},\ \eqref{pi1i'}, and (\ref{pi3}),\ now yield
$$ \Ker \partial^F_3 \ni
\begin{pmatrix}
\pi^3\\[0.1cm]
 (\pi^1_i)_{i=1,\dots,a_1}\\[0.1cm]
(\alpha_\ell)_{\ell=1,\dots,a_2-q_{11}}
\end{pmatrix}=\partial_4^F
\begin{pmatrix}
\pi^4 \\[0.1cm]
(\pi^2_i+\pi'^2_i+p^2_i)_{i=1,\dots,a_1}\\[0.1cm]
 ({\pi'_\ell}^1+p^1_\ell)_{\ell=1,\dots, a_2-q_{11}}\\[0.1cm]
(\alpha'_t)_{t=1,\cdots,a_3-q_{12}}\\[0.1cm]
 (\beta'_s)_{s=1,\dots,a_1^2-q_{11}}
\end{pmatrix}.
$$

\noindent
{\bf Exactness at degree four.} $\im\partial_5^F\subseteq \Ker\partial_4^F$: We show that all three components of the element\\
$\partial_4^F\circ\partial_5^F
\begin{pmatrix}
\pi^5 \\[0.1cm]
(\pi^3_i)_{i=1,\dots,a_1}\\[0.1cm]
(\pi^2_\ell)_{\ell=1,\dots, a_2-q_{11}}\\[0.1cm]
 (\pi^1_t)_{t=1,\dots,a_3-q_{12}}\\[0.1cm]
(\alpha_r)_{r=1,\dots, a_4-a}\\[0.1cm]
(\pi'^1_s)_{s=1,\dots,a_1^2-q_{11}}\\[0.1cm]
(\beta_u)_{u=1,\dots, a_1a_2-a_1q_{11}-q_{12}+b}\\[0.1cm]
(\gamma_{\ell i})_{\ell=1.\dots, a_2-q_{11};\ i=1,\dots, a_1}
\end{pmatrix}
$
are zero. The first component is:
\begin{align*}
&\partial_4^K\Big(
\partial^K_5(\pi^5)
+\hspace{-.1cm}\sum_{i=1}^{a_1}z^1_i\wedge\pi^3_i
-\hspace{-.2cm}\sum_{\ell=1}^{a_2-q_{11}}\hspace{-.2cm}z^2_\ell\wedge\pi^2_\ell
+\hspace{-.2cm}\sum_{t=1}^{a_3-q_{12}}\hspace{-.2cm}z^3_t\wedge  \pi^1_t
+\hspace{-.2cm}\sum_{r=1}^{a_4-a} z^4_r\wedge \alpha_r
-\hspace{-.2cm}\sum_{s=1}^{a_1^2-q_{11}}\hspace{-.1cm}\widetilde \pi^3_s\wedge  \pi'^1_s
-\hspace{-.5cm}\sum_{u=1}^{\tiny\begin{gathered}a_1a_2-a_1q_{11}\vspace{-.1cm}\\-q_{12}+b\end{gathered}}\hspace{-.4cm}\widetilde\pi^4_u\wedge\beta_u\Big)
\\
&\quad
+\sum_{i=1}^{a_1}z^1_i\wedge
\Big(
\partial_3^K(\pi^3_i)
+\sum_{s=1}^{a_1^2-q_{11}}{\widetilde p}^1_{si}\wedge\pi'^1_s
+\sum _{u=1}^{\tiny\begin{gathered}a_1a_2-a_1q_{11}\vspace{-.1cm}\\-q_{12}+b\end{gathered}}\widetilde p^2_{ui}\wedge \beta_u
-\sum_{\ell=1}^{a_2-q_{11}}z^2_\ell\wedge\gamma_{\ell i}
\Big)
\\
&\quad
+\sum_{\ell=1}^{a_2-q_{11}}z^2_\ell\wedge
\Big(
\partial_2^K(\pi^2_\ell)
+\sum_{i=1}^{a_1}z^1_i\wedge \gamma_{i\ell}
\Big)
+\sum_{t=1}^{a_3-q_{12}}z^3_t\wedge\partial^K_1(\pi^1_t)
- \sum_{s=1}^{a_1^2-q_{11}}\widetilde \pi^3_s\wedge \partial_1^K(\pi'_s)
\\
&
= -\sum_{i=1}^{a_1}z^1_i\wedge\partial_3^K(\pi^3_i)
-\sum_{\ell=1}^{a_2-q_{11}}z^2_\ell\wedge\partial_2^K(\pi^2_\ell)
-\sum_{t=1}^{a_3-q_{12}}z^3_t\wedge  \partial_1^K(\pi^1_t)
-\sum_{s=1}^{a_1^2-q_{11}}\partial_3^K(\widetilde \pi^3_s)\wedge  \pi'^1_s
\\
&
\quad
+\sum_{s=1}^{a_1^2-q_{11}}\widetilde \pi^3_s\wedge  \partial_1^K(\pi'^1_s)
-\sum_{u=1}^{\tiny\begin{gathered}a_1a_2-a_1q_{11}\vspace{-.1cm}\\-q_{12}+b\end{gathered}}\partial_4^K(\widetilde\pi^4_u)\wedge\beta_u
+\sum_{i=1}^{a_1}z^1_i\wedge\partial_3^K(\pi^3_i)
+\sum_{s=1}^{a_1^2-q_{11}}\Big(\sum_{i=1}^{a_1}z^1_i\wedge{\widetilde p}^1_{si}\Big)\wedge\pi'^1_s
\\
&\quad
+\sum _{u=1}^{\tiny\begin{gathered}a_1a_2-a_1q_{11}\vspace{-.1cm}\\-q_{12}+b\end{gathered}}\Big(\sum_{i=1}^{a_1}z^1_i\wedge \widetilde p^2_{ui}\Big)\wedge \beta_u
-\sum_{\ell=1}^{a_2-q_{11}}\sum_{i=1}^{a_1}z^1_i\wedge z^2_\ell\wedge\gamma_{\ell i}
+\sum_{\ell=1}^{a_2-q_{11}}z^2_\ell\wedge\partial_2^K(\pi^2_\ell)
\\
& \quad+\sum_{\ell=1}^{a_2-q_{11}}\sum_{i=1}^{a_1}z^2_\ell\wedge z^1_i\wedge \gamma_{\ell i}
+\sum_{t=1}^{a_3-q_{12}}z^3_t\wedge\partial^K_1(\pi^1_t)
-\hspace{-0.3cm}\sum_{s=1}^{a_1^2-q_{11}}\widetilde \pi^3_s\wedge \partial_1^K(\pi'^1_s)=0,
\end{align*}
by definitions of $\widetilde\pi^3_s$,  $\widetilde p^1_{si}$, $\widetilde\pi^4_u$, and $\widetilde p^2_{ui}$.
For each $1\leq i\leq a_1$, as $\widetilde p_{si}^1$ is in $\Ker\partial^K_1$ and $\widetilde p_{ui}^2$ is in $\Ker\partial^K_2$,
we have
\begin{align*}
&\partial_2^K\Big(
\partial_3^K(\pi^3_i)
+\sum_{s=1}^{a_1^2-q_{11}}{\widetilde p}^1_{si}\wedge\pi'^1_s
+\sum _{u=1}^{\tiny\begin{gathered}a_1a_2-a_1q_{11}\vspace{-.1cm}\\-q_{12}+b\end{gathered}}\widetilde p^2_{ui}\wedge \beta_u
-\sum_{\ell=1}^{a_2-q_{11}}z^2_\ell\wedge\gamma_{\ell i}
\Big)
+\sum_{s=1}^{a_1^2-q_{11}}\widetilde p_{si}^1\wedge \partial_1^K(\pi'^1_s)
\\
&=
-\sum_{s=1}^{a_1^2-q_{11}}{\widetilde p}^1_{si}\wedge\partial_1^K(\pi'^1_s)
+\sum_{s=1}^{a_1^2-q_{11}}\widetilde p_{si}^1\wedge \partial_1^K(\pi'^1_s)
= 0.
\end{align*}
Therefore, the second component is zero. For each $1\leq \ell\leq  a_2-q_{11}$, $ \partial_1^K\big(\partial_2^K(\pi^2_\ell)+\sum_{i=1}^{a_1} z^1_i\wedge \gamma_{\ell i}\big)=0.$
Thus, the third component is zero.

$\im\partial_5^F\supseteq \Ker\partial_4^F$:
If $\begin{pmatrix}
\pi^4 \\[0.1cm]
(\pi^2_i)_{i=1,\dots,a_1}\\[0.1cm]
 (\pi^1_\ell)_{\ell=1,\dots, a_2-q_{11}}\\[0.1cm]\
 (\alpha_t)_{t=1,\dots,a_3-q_{12}}\\[0.1cm]
 (\beta_s)_{s=1,\dots,a_1^2-q_{11}}
\end{pmatrix}
\in \Ker\partial_4^F$, then
\begin{align}
\label{eq41}
&\partial^K_4(\pi^4)
+\sum_{i=1}^{a_1}z^1_i\wedge \pi^2_i
+\sum_{\ell=1}^{a_2-q_{11}}z^2_\ell\wedge \pi^1_\ell
+\sum_{t=1}^{a_3-q_{12}}z^3_t\wedge\alpha_t
-\sum_{s=1}^{a_1^2-q_{11}}{\widetilde\pi}^3_s\wedge\beta_s =0,\\
\label{eq42}
&\partial_2^K(\pi^2_i)+\sum_{s=1}^{a_1^2-q_{11}}{\widetilde p}^1_{si}\wedge \beta_s=0
 \quad \mbox{for all}\ 1\leq i\leq a_1,\quad\text{and}
\\
\label{eq43}
&\partial^K_1(\pi^1_\ell)=0 \quad \mbox{for all}\ 1\leq \ell\leq a_2-q_{11}.
\end{align}
The equality \eqref{eq43} implies that there exist $\pi'^2_\ell \in K_2$ and $\gamma_{i \ell}\in R$ such that
\begin{equation}
\label{pi1l}
\pi^1_\ell=\partial_2^K(\pi'^2_\ell)+\sum_{i=1}^{a_1}z^1_i\wedge \gamma_{i\ell}.
\end{equation}
In $A_1^{a_1}$ the equality (\ref{eq42}) becomes $\sum_{s=1}^{a_1^2-q_{11}}([\widetilde p^1_{s1},\dots,[\widetilde p^1_{sa_1}])\wedge [\beta_s]=0.$ Applying Lemma \ref{ind p1} we obtain $[\beta_s]=0$ for all $s$, thus
\begin{equation}
\label{betas}
\beta_s = \partial_1^K(\pi'^1_s) \quad \text{ for some } \pi'^1_s \in K_1.
\end{equation}
The equality (\ref{eq42}) now becomes
$$\partial_2^K(\pi^2_i) + \sum_{s=1}^{a_1^2-q_{11}} {\widetilde p}^1_{si}\wedge \partial_1^K(\pi'^1_s)=
\partial_2^K \Big( \pi^2_i - \sum_{s=1}^{a_1^2-q_{11}}{\widetilde p}^1_{si} \wedge  \pi'^1_s \Big)= 0.
$$
Therefore, there exist  $\pi^3_i$ in $K_3$, $\delta_{ijk}$ and $\gamma'_{i\ell}$ in $R$ such that  for all $i$ we have:
\begin{equation}
\label{pi2i}
\pi^2_i=\partial_3^K(\pi^3_i)+ \sum_{j,k=1}^{a_1} z^1_j\wedge z^1_k\wedge\delta_{ijk}+\sum_{\ell=1}^{a_2-q_{12}}z^2_\ell\wedge\gamma'_{i\ell} + \sum_{s=1}^{a_1^2-q_{11}} {\widetilde p}^1_{si}\wedge \pi'^1_s.
\end{equation}
Putting together \eqref{pi1l}, \eqref{betas}, and \eqref{pi2i}, into the equality \eqref{eq41} we get:
\begin{align*}
0&=\partial^K_4(\pi^4)
  +\sum_{i=1}^{a_1}z^1_i\wedge \Big(\sum_{s=1}^{a_1^2-q_{11}} {\widetilde p}^1_{si}\wedge \pi'^1_s + \partial_3^K(\pi^3_i)+ \sum_{j,k=1}^{a_1} z^1_j\wedge z^1_k\wedge\delta_{ijk}+\sum_{\ell=1}^{a_2-q_{12}}z^2_\ell\wedge\gamma'_{i\ell}\Big)
  \\
  &
 \quad+\sum_{\ell=1}^{a_2-q_{11}}z^2_\ell\wedge \Big(\partial_2^K({\pi'_\ell}^2)+\sum_{i=1}z^1_i\wedge \gamma_{i\ell}\Big)
+\sum_{t=1}^{a_3-q_{12}}z^3_t\wedge\alpha_t
-\sum_{s=1}^{a_1^2-q_{11}}{\widetilde\pi}^3_s\wedge\partial_1^K(\pi'^1_s)\\
&=\partial^K_4\Big(\pi^4 + \sum_{s=1}^{a_1^2-q_{11}} \widetilde\pi^3_s\wedge\pi'^1_s-\sum_{i=1}^{a_1}z^1_i\wedge\pi^3_i+\sum_{\ell=1}^{a_2-q_{11}}z^2_\ell\wedge{\pi'_\ell}^2\Big)\\
    &\quad
    +\Big(\sum_{i,j,k=1}^{a_1}z^1_i\wedge z^1_j\wedge z^1_k\wedge\delta_{ijk}
    +\sum_{i=1}^{a_1}\sum_{\ell=1}^{a_2-q_{11}}z^1_i\wedge z^2_\ell\wedge (\gamma'_{i\ell}+\gamma_{i\ell})
    +\sum_{t=1}^{a_3-q_{12}}z^3_t\wedge\alpha_t\Big).\nonumber
  \end{align*}
In $A_3$ this reduces to:
$$\sum_{i,j,k=1}^{a_1}[z^1_i]\wedge [z^1_j]\wedge [z^1_k]\wedge[\delta_{ijk}]
+\sum_{i=1}^{a_1}\sum_{\ell=1}^{a_2-q_{11}}[z^1_i]\wedge [z^2_\ell]\wedge [\gamma'_{i\ell}+\gamma_{i\ell}]
+\sum_{t=1}^{a_3-q_{12}}[z^3_t]\wedge[\alpha_t]=0.$$
By Proposition \ref{z1z2z3}, for each $t$ there exists $\pi''^1_t\in K_1$, and for each $i$ there exist elements $\delta'_{u}, \varepsilon_{js}$ in $R$, and $\pi'^3_i\in K_3$ such that

\begin{align}
\label{alphat}
\alpha_t&=\partial_1^K({\pi''_t}^1),  \text{ and } \\
\sum_{j,k=1}^{a_1} z^1_j \wedge z^1_k \wedge \delta_{ijk}
+ \sum_{\ell=1}^{a_2-q_{11}} z^2_\ell \wedge (\gamma_{i\ell} + \gamma'_{i\ell})
&= \partial_3^K({\pi_i'}^3) + \hspace{-0.2cm} \sum_{u=1}^{\tiny\begin{gathered}a_1a_2 - a_1q_{11}
\vspace{-0.1cm}\\ -q_{12}+b \end{gathered}}\hspace{-0.2cm} \widetilde p^2_{ui} \wedge \delta'_u
+\sum_{j=1}^{a_1}\sum_{s=1}^{a_1^2-q_{11}} \widetilde p^1_{si} \wedge z^1_j \wedge \varepsilon_{js} \nonumber.
\end{align}
Thus, \eqref{eq41} further becomes:
\begin{align*}
0&=\partial^K_4\Big(\pi^4 + \sum_{s=1}^{a_1^2-q_{11}} \widetilde\pi^3_s\wedge \pi'^1_s-\sum_{i=1}^{a_1}z^1_i\wedge\pi^3_i+\sum_{\ell=1}^{a_2-q_{11}}z^2_\ell\wedge {\pi'_\ell}^2\Big)\\
&\quad + \sum_{i=1}^{a_1}z^1_i\wedge\Big(\partial_3^K({\pi'_i}^3)+\sum_{u=1}^{\tiny\begin{gathered}a_1a_2-a_1q_{11}
\vspace{-0.1cm}\\
-q_{12}+b\end{gathered}} \widetilde p^2_{ui} \wedge {\delta'_u}
+\sum_{j=1}^{a_1}\sum_{s=1}^{a_1^2-q_{11}} \widetilde p^1_{si} \wedge z^1_j \wedge \varepsilon_{js}\Big) +\sum_{t=1}^{a_3 - q_{12}} z^3_t\wedge \partial_1^K({\pi''_t}^1)\\
&=\partial^K_4\Big(\pi^4 + \sum_{s=1}^{a_1^2-q_{11}} \widetilde\pi^3_s\wedge \pi'^1_s-\sum_{i=1}^{a_1}z^1_i\wedge\pi^3_i+\sum_{\ell=1}^{a_2-q_{11}}z^2_\ell\wedge \pi'^2_\ell\\
&\quad-\sum_{i=1}^{a_1}z^1_i\wedge {\pi'_i}^3+\sum_{u=1}^{\tiny\begin{gathered}a_1a_2-a_1q_{11}
\vspace{-0.1cm}\\
-q_{12}+b\end{gathered}} \widetilde \pi^4_{u} \wedge \delta'_u+\sum_{j=1}^{a_1}\sum_{s=1}^{a_1^2-q_{11}} \widetilde \pi^3_{s} \wedge z^1_j \wedge \varepsilon_{js}-\sum_{t=1}^{a_3-q_{12}}z^3_t\wedge {\pi''_t}^1\Big).
\end{align*}

The second equality above follows from definitions of $\widetilde\pi^3_s$ and $\widetilde\pi^4_u$ in  \eqref{def pi3} and \eqref{def pi4}. Therefore, there exist $\pi^5\in K_5$, $\delta_r\in R$,  $p^3_i\in\Ker\partial^K_3,  p^2_\ell\in\Ker\partial^K_2$ and $p^1_s\in\Ker\partial^K_1$  with $p^1_s$ as in Proposition~\ref{massey}, such that
\begin{align*}
&\pi^4 - \sum_{i=1}^{a_1}z^1_i\wedge (\pi^3_i+ {\pi'_i}^3) + \sum_{\ell=1}^{a_2-q_{11}}z^2_\ell\wedge {\pi'_\ell}^2 -\sum_{t=1}^{a_3-q_{12}}z^3_t\wedge{\pi''_t}^1  \\
&\qquad + \sum_{s=1}^{a_1^2-q_{11}} \widetilde\pi^3_s\wedge \Big(\pi'^1_s + \sum_{j=1}^{a_1}z^1_j \wedge \varepsilon_{js} \Big) +\sum_{u=1}^{\tiny\begin{gathered}a_1a_2-a_1q_{11} \vspace{-0.1cm}\\-q_{12}+b\end{gathered}} \widetilde \pi^4_{u} \wedge {\delta'_u}  \\
&= \partial_5^K(\pi^5)+  \sum_{i=1}^{a_1} z^1_i \wedge p^3_i+  \sum_{\ell=1}^{a_2-q_{11}} z^2_\ell \wedge p^2_\ell +\sum_{s=1}^{a_1^2-q_{11}} \widetilde \pi^3_s \wedge p^1_s +\sum_{r=1}^{a_4-a} z^4_r \wedge \delta_r.
\end{align*}
This implies:
\begin{equation}
  \label{pi4}
  \begin{aligned}
\pi^4&= \partial_5^K(\pi^5)
+\sum_{i=1}^{a_1} z^1_i \wedge (\pi^3_i+ {\pi'_i}^3 + p^3_i)
-\sum_{\ell=1}^{a_2-q_{11}} z^2_\ell \wedge ( {\pi'_\ell}^2 - p^2_\ell)
+ \sum_{t=1}^{a_3-q_{12}}z^3_t\wedge {\pi''_t}^1
+\sum_{r=1}^{a_4-a} z^4_r \wedge \delta_r \\
&\quad - \sum_{s=1}^{a_1^2-q_{11}} \widetilde\pi^3_s\wedge \Big(\pi'^1_s
+ \sum_{j=1}^{a_1}z^1_j \wedge \varepsilon_{js} - p^1_s \Big)
- \sum_{u=1}^{\tiny\begin{gathered}a_1a_2-a_1q_{11} \vspace{-0.1cm}\\-q_{12}+b\end{gathered}} \widetilde \pi^4_{u} \wedge \delta'_u.\\
\end{aligned}
\end{equation}
By \eqref{alphat}, the expression \eqref{pi2i} becomes:
\begin{align}
 \label{pi2}
\pi^2_i&=\partial_3^K(\pi^3_i)
+\sum_{j,k=1}^{a_1} z^1_j\wedge z^1_k\wedge\delta_{ijk}
+\sum_{\ell=1}^{a_2-q_{12}}z^2_\ell\wedge{\gamma'_{i\ell}}
+\sum_{s=1}^{a_1^2-q_{11}} {\widetilde p}^1_{si}\wedge \pi'^1_s
 \\
 &=\partial_3^K(\pi^3_i+{\pi'_i}^3)+ \sum_{s=1}^{a_1^2-q_{11}} {\widetilde p}^1_{si}\wedge \Big(\pi'^1_s + \sum_{j=1}^{a_1} z^1_j \wedge \varepsilon_{js} \Big) + \hspace{-0.2cm} \sum_{u=1}^{\tiny\begin{gathered}a_1a_2 - a_1q_{11}
\vspace{-0.1cm}\\ -q_{12}+b \end{gathered}} \widetilde p^2_{ui} \wedge\delta'_u - \sum_{\ell=1}^{a_2-q_{11}} z^2_\ell \wedge \gamma_{i\ell} \nonumber.
\end{align}
We conclude that by using \eqref{pi1l}, \eqref{betas}, \eqref{alphat}, \eqref{pi4}, \eqref{pi2}, and Proposition~\ref{massey}, the chosen kernel element is in the image of $\partial^F_5$:

\begin{equation*}
\Ker\partial^F_4 \ni
\begin{pmatrix}
\pi^4 \\[0.1cm]
(\pi^2_i)_{i=1,\dots,a_1}\\[0.1cm]
 (\pi^1_\ell)_{\ell=1,\dots, a_2-q_{11}}\\[0.1cm]\
 (\alpha_t)_{t=1,\dots,a_3-q_{12}}\\[0.1cm]
 (\beta_s)_{s=1,\dots,a_1^2-q_{11}}
 \end{pmatrix}
  =
\partial_5^F
\begin{pmatrix}
\pi^5 \\[0.1cm]
(\pi^3_i+{\pi'_i}^3 + p^3_i)_{i=1,\dots,a_1}\\[0.1cm]
(\pi'^2_\ell - p^2_\ell) _{\ell=1,\dots, a_2-q_{11}}\\[0.1cm]
 (\pi''^1_t)_{t=1,\dots,a_3-q_{12}}\\[0.1cm]
(\delta_r)_{r=1,\dots, a_4-a}\\[0.1cm]
\Big(\pi'^1_s +
\sum_{j=1}^{a_1}z^1_j \wedge \varepsilon_{js} - p^1_s \Big)_{s=1,\dots,a_1^2-q_{11}} \\[0.1cm]
(\delta'_u)_{u=1,\dots, a_1a_2-a_1q_{11}-q_{12}+b}\\[0.1cm]
(\gamma_{\ell i})_{\ell=1,\dots, a_2-q_{11};\ i=1,\dots, a_1}
\end{pmatrix}.
\end{equation*}
\end{proof}

\begin{remark}
If all multiplications on the algebra  $A$ are trivial, that is, $q_{ij}=0$  for all $i,j\geq 1$, and all Massey operations are zero, Golod's construction in \cite{G} gives a minimal free resolution of the residue field $\kk$ of a local ring $R$ in terms of the
  Koszul complex $K$ of $R$. Without these assumptions, our complex $F$  in Construction \ref{construction} generalizes Golod's resolution up to degree five.
\end{remark}


\section{Applications of the Construction}
\label{sec applications}
For a local ring $R$ of embedding dimension $n$ and codepth $c$, set $\beta_i:=\rank_{\kk}\Tor{i}{R}{\kk}{\kk}$ for $i\geq 0$ for the Betti numbers of $\kk$ over $R$. Let $\sum_{i=0}^\infty b_i t^i$ denote the series on the right hand side of the inequality \eqref{golodP}. The sequence  $\CP:=\{\CP_i\}_{i\geq 0}$ defined by:
 \begin{equation}
 \label{def CP}
 \CP_i:=b_i-\beta_i, \quad\text{for all } i\geq 0,
 \end{equation}
gives the coefficients of the series $\CP(t)$ defined in the Introduction. The ring $R$ is Golod if and only if $\CP_i=0$ for all $i\geq 0$. Our goal in this section is to give a description of the sequence $\{\CP_i\}_{0\leq i\leq 5}$
in terms of the invariants of the multiplicative structure of the algebra $A=\HH{}{K^R}$ and discuss various consequences of Theorem~\ref{main}.
First, using Theorem \ref{main}, we explicitly describe the Betti numbers $\beta_i$ in terms of those invariants, up to degree five.

\begin{corollary}
\label{betti}
Let $R$ be a local ring of embedding dimension $n$. Let $a_i$, $q_{ij}$, $a$ and $b$ be as in Summary~\ref{summary}. Then the following equalities hold:
   \begin{equation*}
    \beta_0=1,\qquad \beta_1=n,\qquad \beta_2 =\binom n2+a_1,
\end{equation*}
 \begin{equation*}
\beta_3={n\choose 3}+na_1+a_2-q_{11}, \qquad \beta_4={n\choose 4}+{n\choose 2}a_1+ n a_2 + a_3 + a_1^2-(n+1)q_{11}-q_{12},
\end{equation*}
 \begin{equation*}
\beta_5={n\choose 5}+{n\choose 3}a_1+{n\choose 2}a_2 + na_3+a_4 +na_1^2 +2a_1a_2-\Big({{n+1} \choose 2}+2a_1\Big) q_{11}-(n+1)q_{12}+b-a.
\end{equation*}
\end{corollary}

\begin{proof} This is a direct consequence of Theorem \ref{main}. The formulas for $\beta_0,\dots,\beta_3$ are clear, and we simplify the following expressions for $\beta_4$ and $\beta_5$ to obtain the ones in the statement:
  \begin{align*}
  \beta_4&={n\choose 4}+{n\choose 2}a_1+a_1^2-q_{11}+n(a_2-q_{11})
+a_3-q_{12}\\
\beta_5&={n\choose 5}+{n\choose 3}a_1+n(a_1^2-q_{11})+a_1a_2-a_1q_{11}-q_{12}+b\\
&\hspace{0.5cm}+{n\choose 2}(a_2-q_{11})+a_1a_2-a_1q_{11}+n(a_3-q_{12})+a_4-a.
\end{align*}
\end{proof}

 \begin{proposition}
 \label{prop:P}
 Let $R$ be a local ring of embedding dimension $n$ and codepth $c$. Let $a_i$, $q_{ij}$, $a$, and $b$ be as in Summary~\ref{summary}. The sequence $\CP$ from \eqref{def CP} satisfies the following equalities:
$$ \CP_0=\CP_1=\CP_2=0,\qquad  \CP_3=q_{11},\qquad \CP_4=(n+1)q_{11}+q_{12},\qquad\text{and}$$
 $$\CP_5=\Big({{n+1} \choose 2}+2a_1\Big)q_{11}+(n+1)q_{12}+a-b.$$
 \end{proposition}

 \begin{proof} 
 Recall that the right hand side of the inequality \eqref{golodP} is $\sum_{i=1}^{\infty}b_it^i=\frac{(1+t)^n}{1 -\sum_{i=1}^c a_i t^{i+1}}$.
   Comparing the coefficients on both sides of \eqref{golodP} we have the following recursive formulas for $b_i$:
   \begin{equation*}
   b_0=1, \qquad b_1=n,\qquad b_i=\sum_{j=1}^{i-1}a_jb_{i-j-1}+{n\choose i},\quad\mbox{for all}\  i\geq 2.
 \end{equation*}
In particular, we obtain:
   \begin{align*}
   b_2&={n\choose 2}+a_1, & b_4&={n\choose 4}+{n\choose 2}a_1+ n a_2 + a_3 + a_1^2,\\
   b_3&={n\choose 3}+na_1+a_2,&b_5&={n\choose 5}+{n\choose 3}a_1+{n\choose 2}a_2 + na_3+a_4 +na_1^2 +2a_1a_2.
 \end{align*}
Now, the expression for $\CP_i=b_i-\beta_i$ for $0\leq i\leq 5$ follows from Corollary \ref{betti}.

Alternatively, assume  $q_{ij}=0$ for all $1\leq i+j\leq 4$, and Massey products are also trivial, we have $a=b=0$.  In Construction \ref{construction}  we obtain the maximum possible values of the Betti numbers, so the differences between the actual Betti numbers and these maximum values are exactly the $\CP_i$'s in the proposition.
\end{proof}

\begin{corollary}
\label{cor:P}
Let $R$ be a local ring, $a_i$, $q_{ij}$, $a$, and $b$ be as in Summary~\ref{summary} and $\CP$ as in \eqref{def CP}. 
\begin{enumerate}[$(a)$]
    \item If $q_{11}=q_{12}=0$, then $\CP_i=0$ for all $0\leq i\leq 4$ and $\CP_5=a$.  
    \item If $q_{11}=q_{12}=q_{13}=q_{22}=0$, then $\CP_5=\rank_\kk \big(\Span_\kk\langle A_1, A_1, A_1\rangle\big)$.
    \item If $\codepth R=4$, then $R$ is Golod if and only if $\CP_i=0$ for all $1\leq i\leq 5$.
\end{enumerate}
\end{corollary}
\begin{proof} (a) and (b) are straightforward from  Proposition~\ref{prop:P}. For (c), result in \cite{G} showed that $R$ is Golod if and only if all products on $A$ and all ternary Massey products are trivial. If $\codepth R=4$, then the $r$-ary Massey products are trivial for all $r\geq 4$. Thus, part (c) follows from Proposition~\ref{prop:P}.
\end{proof}

\begin{remark}
A result of  Burke  \cite[Corollary 6.10]{B} implies  that  for a local ring $R$ of codepth $c$ the following implication holds:
$$\text{If}\  \CP_i=0\ \text{for all}\  0\leq i\leq c+1,\ \text{then}\ \CP_i=0\ \text{for all}\  i\geq 0.$$
Thus, Corollary~\ref{cor:P}(c) is  a consequence of this result as well.
\end{remark}

We now examine some local rings with rational Poincar\'e series of certain forms and describe $\Poi{R}{\kk}(t)$ in terms of the algebraic invariants of $A$.
\begin{proposition}
  \label{Rational Poincare}
  Let $(R,\fm, \kk)$ be a local ring of embedding dimension $n$ and codepth $c$. Let $a_i, q_{ij}, a,$ and $b$ be as in Summary \ref{summary}.
If the  Poincar\'e series of $R$ is rational of the form $\Poi{R}{\kk}(t)=(1+t)^n/d(t),$  then 
$$d(t)=\Big(1-\sum_{i=1}^c a_it^{i+1}\Big)+ q_{11}t^3+(q_{11}+q_{12})t^4+(q_{12}-b+a)t^5 +f(t)t^6,$$
  for some $f(t)\in\BZ[t]$.
\end{proposition}

\begin{proof} Let  $\CP_i$ be as in \eqref{def CP}. Set
$$\CP(t)=\sum_{i=0}^\infty\CP_it^i,\qquad \alpha(t)=1-\sum_{i=1}^c a_it^{i+1}\qquad  \text{and}\qquad \gamma(t)=d(t)-\alpha(t).$$
Then
$$\CP(t)=\frac{(1+t)^n}{\alpha(t)}-\frac{(1+t)^n}{d(t)}=\frac{(1+t)^n\cdot \gamma(t)}{\alpha(t)\cdot(\alpha(t)+\gamma(t))}\iff$$
$$\frac{\gamma(t)}{\alpha(t)+\gamma(t)}=\frac{\CP(t)\cdot\alpha(t)}{(1+t)^n}\iff \frac{\gamma(t)}{\alpha(t)}=\frac{\CP(t)\cdot\alpha(t)}{(1+t)^n-\CP(t)\cdot\alpha(t)}\iff
  \gamma(t)=\frac{\CP(t)\cdot(\alpha(t))^2}{(1+t)^n-\CP(t)\cdot\alpha(t)}.
$$
We compare the coefficients of $t^i$ for all $0\leq i\leq 5$ on both sides of the following equality
\begin{equation*}
  \label{eq gamma}
  \gamma(t)\cdot\Big((1+t)^n-\CP(t)\cdot\alpha(t)\Big)= \CP(t)\cdot(\alpha(t))^2.
\end{equation*}
By Proposition \ref{prop:P}, $\CP_0=\CP_1=\CP_2=0$, and thus the left and right hand sides of the above equation become:
\begin{align*}
  \text{LHS}& = (\gamma_0+\gamma_1t+\gamma_2t^2+\gamma_3t^3+\gamma_4t^4+\gamma_5t^5+\cdots)\\
            &\quad\cdot \left(1+nt+{n\choose 2}t^2+\left({n\choose 3}-\CP_3\right)t^3+\left({n\choose 4}-\CP_4\right)t^4+\left({n\choose 5}+\CP_3a_1-\CP_5\right)t^5+\cdots\right),\\
  \text{RHS}&=(\CP_3t^3+\CP_4t^4+\CP_5t^5+\dots)\cdot\left(1-2a_1t^2-2a_2t^3+(a_1^2-2a_4)t^4+(a_1a_2-2a_5)t^5+\cdots\right).
\end{align*}
It is clear that $\gamma_0=\gamma_1=\gamma_2=0$ and comparing the coefficients of $t^3$ we get:
$\gamma_3=\CP_3=q_{11}.$
Comparing the coefficients of $t^4$ and by Proposition \ref{prop:P} we get:
$$ \gamma_4+n\gamma_3=\CP_4\iff \gamma_4=\CP_4-n\gamma_3=(n+1)q_{11}+q_{12}-nq_{11}=q_{11}+q_{12}.$$
Finally, comparing the coefficients of $t^5$  and by Proposition \ref{prop:P} we get:
$$  \gamma_5+n\gamma_4+{n\choose 2}\gamma_3=\CP_5-2a_1\CP_3\iff$$
\begin{align*}
\gamma_5&=\CP_5-2a_1\CP_3-n\gamma_4-{n\choose 2}\gamma_3\\
          &=\Big({{n+1} \choose 2}+2a_1\Big)q_{11}+(n+1)q_{12}-b+a-2a_1q_{11}-n(q_{11}+q_{12})-{n\choose 2}q_{11}\\
          &=q_{12}-b+a.
\end{align*}
Therefore, the expression for $d(t)=\alpha(t)+\gamma(t)$ in the statement holds.
\end{proof}

There are many classes of local rings for which the Poincar\'e series is rational of the form $\Poi{R}{\kk}(t)=(1+t)^n/d(t).$
In light of Proposition \ref{Rational Poincare}, we write the coefficients of the polynomial $d(t)$ in terms of the invariants $a_i, q_{ij}, b$ and $a$ for some special cases and provide a {\it uniform expression} of the Poincar\'e series in these cases.

\begin{corollary}
  \label{codepth3}
  If $(R,\fm,\kk)$ is a  local ring of embedding dimension $n$ and codepth at most 3, and $a_i, q_{ij},b$ are as in Summary \ref{summary}, then
  \begin{equation*}
    \Poi{R}{\kk}(t)=\frac{(1+t)^{n-1}}{1-t-(a_1-1)t^2-(a_3-q_{11})t^3+ q_{12}t^4- bt^5}.
  \end{equation*}
\end{corollary}

\begin{proof}
   By \cite[Theorem 3.5]{Av0}
the Poincar\'e series of the ring $R$ is rational given by $\Poi{R}{\kk}(t)=(1+t)^n/d(t)$ with $\deg d(t)\leq 6$ and $d(-1)=0$.  Since $\codepth R \leq 3$ we have $a_4=0$ and $a=0$. Thus, Proposition \ref{Rational Poincare} gives:
  $$P^R_\kk(t)=\frac{(1+t)^{n}}{1-a_1t^2-(a_2-q_{11})t^3-(a_3-q_{11}-q_{12})t^4 -(b-q_{12})t^5- bt^6}.$$
  Simplifying the fraction by the common factor $(1+t)$, we get the desired conclusion.
\end{proof}

\begin{remark} In \cite[Lemma 3.6]{Av2} Avramov defined the invariant $\tau$ for non-Gorenstein ring $R$ of codepth 3 as follows:   $\tau=1$ if  $R$ is of  class  $\clT$, and $\tau =0$ otherwise.
  By comparing \cite[(3.6.2)]{Av2} and the Poincar\'e expression in Corollary \ref{codepth3} we see that $\tau$ is our $b=\rank_{\kk}(\Coker\psi)$ in \eqref{def b}, which is a multiplicative invariant of the homology algebra $A$.
It follows from  \cite[Theorem 3.5]{Av0} that  $b=1$ for a ring of class $\clC3$ or $\clT$, and $b=0$ for rings of class  $\clC1,\  \clC2,\  \clS,\  \clB,\  \clG{r},\ \text{and}\  \clH{q_{11},q_{12}}$.
\end{remark}


\begin{corollary}
  \label{Gor codepth4}
  If $(R,\fm,\kk)$ is a  Gorenstein local ring, not a complete intersection, of embedding dimension $n$ and  codepth 4, and and $a_i, q_{ij}, a, b$ are as in Summary \ref{summary}, then

  \begin{equation*}
    \Poi{R}{\kk}(t)=\frac{(1+t)^{n-2}}{1-2t-(a_1-3)t^2+(q_{11}-2)t^3+ (-q_{11}+q_{12}-1)t^4-(q_{11}-q_{12}-b+a-1)t^5}.
  \end{equation*}
\end{corollary}

\begin{proof} Since $R$ is Gorenstein we have $a_2=2a_1-2, a_3=a_1,\ \text{and}\  a_4=1.$
  The statement follows from \cite[Theorem 3.5]{Av0} and Proposition \ref{Rational Poincare}.
\end{proof}

\begin{remark} Using the Avramov's notation from \cite{Av0} for the classification of Gorenstein local rings $R$ of codepth 4 given by Kustin and Miller \cite{KM85}  and comparing the Poincar\'e expressions from  \cite[Theorem 3.5]{Av0} and Corollary \ref{Gor codepth4}, we obtain the following  table of algebra invariants of a Gorenstein local ring $R$:
    \begin{equation*}
      \begin{array}{r|cccccc}
        \text{Class}& q_{11}&q_{12}&q_{22}&q_{13}&a&b\\
        \hline
        \clC{4} &6&4&1&1&1&4\\
        \clGT   &3&3&1&1&1&1\\
        \clGS   &0&0&1&1&1&0\\
        \clGH{p}&p&p+1&1&1&1&0
           \end{array}
      \end{equation*}
    \end{remark}

We provide more examples of rings for which we can calculate the multiplicative invariants $q_{ij}, a, b$ from the Poincar\'e series of the ring.

\begin{corollary}
  \label{yoshino}
  Let $\kk$ be a field, $I$ be an ideal of $\kk[x,y,z,w]$ such that $(x,y,z,w)^3\subseteq I\subseteq (x,y,z,w)^2$, and set $R=\kk[x,y,z,w]/I$. Let $a_i, q_{ij},a,b$ be as in Summary \ref{summary}. If $R$ has $a_4=3$ and its Poincar\'e series is of the form $\Poi{R}{\kk}(t)=\frac{1}{(1-t)(1-3t)}$, then
  $$q_{11}=a_2-8,\qquad q_{12}=8,\qquad a=3,\qquad \text{and}\qquad b=0.$$
\end{corollary}

\begin{proof} By hypothesis,  $$\Poi{R}{\kk}(t)=\frac{(1+t)^4}{(1-t)(1-3t)(1+t)^4}=\frac{(1+t)^4}{1-7t^2-8t^3+3t^4+8t^5+3t^6}.$$
Proposition \ref{Rational Poincare} thus gives:
\begin{equation}
  \label{aq eq}
a_1=7,\quad a_2-q_{11}=8, \quad a_3-q_{11}-q_{12}=-3,\quad\text{and}\quad a_4-q_{12}+b-a=-8.
\end{equation}
From the second equality we obtain $q_{11}=a_2-8$ and from the third we get $q_{12}=a_3-a_2+11$.
Since $a_4=3$ and $a_1=7,$ we get $a_2=a_3+3$ and thus $q_{12}=8$.
The last equality in \eqref{aq eq} becomes $a-b=3$.
 By the definitions of $a$ and $b$ we have $0\leq a\leq a_4=3$ and $b\geq 0$, thus $a=3$ and $b=0$.
\end{proof}

\begin{examples}
\label{exp:yoshino}
Rings satisfying the hypotheses of Corollary \ref{yoshino} are discussed by Yoshino in \cite{Y} and Christensen and Veliche in \cite{CV}.
For examples, consider the following ideals in $Q=\BQ[x,y,z,w]$:
\begin{align*}
  I_1&=(yw, xw+zw+w^2, z^2+w^2, xz+zw+w^2, y^2+yz, xy+zw, x^2+zw), \\ 
  I_2&=(zw + w^2 , yw, z^2  + w^2 , yz + xw + w^2 , xz + w^2 , xy + y^2  + xw + w^2 , x^2 + xw + w^2), \\ 
  I_3&=(zw, yw, xw-w^2, yz, xz, xy-z^2, x^2-y^2), \\ 
  I_4&=(w^2 , yw + zw, xw, yz + z^2 , y^2  + zw, xy + xz, x^2  + zw). 
\end{align*}
  The rings $Q/I_i$ satisfy the hypotheses of Corollary \ref{yoshino} and have $a_2=10+i$, for $1\leq i\leq 4$.
\end{examples}

\begin{examples}
\label{roos}
In an unpublished note,  Roos \cite{R}, inspired by a paper of Katth\"an \cite{K}, constructed several examples of non-Golod rings $R$ of codepth 4 with trivial algebra multiplications on $A$. We provide here two of them. For each one, there exists a Golod homomorphism from a complete intersection ring, hence it has rational Poincar\'e series. The algebra multiplication on $A$  was checked using the {\tt DGAlgebras} package \cite{DGA} of {\tt Macaulay2} \cite{M2}.  Proposition \ref{Rational Poincare} confirms that indeed $q_{11}=q_{12}=0$, and moreover it gives us the exact size of the space generated by the triple Massey products $\langle A_1,A_1,A_1\rangle$, known to be nonzero.

Consider the following ideals in $Q=\BQ[x,y,z,w]$:
\begin{align*}
  J_1&=(w^3, xy^2, xz^2+yz^2,  x^2w, x^2y+y^2w, y^2z+z^2w), \\ 
  J_2&=(w^3, xy^2, xz^2+yz^2, x^2w+zw^2, y^2w+xzw, y^2z+yz^2). 
\end{align*}
The rings $Q/J_i$ with $i=1,2$ are non-Artinian of codepth 4 with
\begin{align*}
  \Poi{Q}{Q/J_i}(t)&=1+6t+(10+i)t^2+(7+i)t^3+2t^4,\\
  \Poi{Q/J_i}{\kk}(t)&=\frac{(1+t)^4}{1-6t^2-(10+i)t^3-(7+i)t^4-t^5+t^6}.
\end{align*}
For both rings $Q/J_i$,  $q_{11}=0=q_{12}$ by Proposition \ref{Rational Poincare}, hence $b=0$ and  $a=1$.
Since $q_{22}=q_{13}=0$, the space spanned by the ternary Massey products  $\langle A_1, A_1, A_1\rangle$ has rank one.
\end{examples}

By \cite[Example 7.1]{A}, not all local rings have rational Poincar\'e series. However, for every local ring $R$ with residue field $\kk$, there exists a unique sequence of integers $\{\varepsilon_i\}_{i\geq 0}$ such that the Poincar\'e series of $R$ can be expressed as
 \begin{equation*}
 P_\kk^R(t)=\frac{\prod_{i=1}^{\infty}(1+t^{2i-1})^{\varepsilon_{2i-1}}}{\prod_{i=1}^{\infty}(1-t^{2i})^{\varepsilon_{2i}}},
 \end{equation*}
 and $\varepsilon_i$ is called the \emph{$i$-th deviation of $R$}; see for example \cite[Remark 7.1.1]{Av1}. Theorem~\ref{main} allows us to describe the first five  deviations in terms of the algebraic invariants of $A$.
\begin{corollary} 
\label{deviations}
Let $R$ be a local ring of embedding dimension $n$ and $a_i, q_{ij}$ be as in Summary~\ref{summary}. Then the first five deviations of $R$ are
 \begin{align*}
   \varepsilon_1&=n,\qquad  \varepsilon_2 =a_1,\qquad \varepsilon_3=a_2-q_{11},\\
   \varepsilon_4&=a_3-q_{12}+{a_1\choose 2}-q_{11},\\
   \varepsilon_5&=a_4+a_1a_2-a_1q_{11}-q_{12}+b-a.
 \end{align*}
 \end{corollary}
 \begin{proof} 
 By comparing the coefficients on the left and right sides of the equality
$$\prod_{i=1}^{\infty}(1+t^{2i-1})^{\varepsilon_{2i-1}}=\prod_{i=1}^{\infty}(1-t^{2i})^{\varepsilon_{2i}}\cdot (\beta_0+\beta_1t+\beta_2t^2+\beta_3t^3+\beta_4t^4+\beta_5t^5+\cdots)$$
one obtains the following relations between the Betti numbers $\{\beta_i\}_{1\leq i\leq 5}$ and the deviations $\{\varepsilon_i\}_{1\leq i\leq 5}$:
\begin{align*}
 \beta_1&=\varepsilon_1,\qquad \beta_2=\varepsilon_2+{\varepsilon_1\choose 2},\qquad \beta_3=\varepsilon_3+\varepsilon_2\varepsilon_1+{\varepsilon_1\choose 3},\\
 \beta_4&=\varepsilon_4+\varepsilon_3\varepsilon_1+{1+\varepsilon_2\choose 2}+\varepsilon_2{\varepsilon_1\choose 2}+{\varepsilon_1\choose 4},\\
 \beta_5
 &=\varepsilon_5
 +\varepsilon_4\varepsilon_1
 +\varepsilon_3\varepsilon_2
 +\varepsilon_3{\varepsilon_1\choose 2}
 +\varepsilon_2^2\varepsilon_1
 -\varepsilon_1{\varepsilon_2\choose 2}
 +\varepsilon_2{\varepsilon_1\choose 3}
 +{\varepsilon_1\choose 5}.
\end{align*}
The first four relations were also given by Avramov \cite[page 62]{Av1} and Gulliksen and Levin \cite[Proposition 3.3.4, Theorem 4.4.3]{GL}. Note that we have corrected the expression for $\beta_ 4$, compared to that given in \cite[page 62]{Av1}. The expressions for $\varepsilon_i$ described in the statement follow from these relations and Corollary~\ref{betti}. 
\end{proof}

\begin{remark} 
The formulas for $\varepsilon_2, \varepsilon_3, \varepsilon_4$ in Corollary \ref{deviations}  were previously given in \cite[Corollary 6.2]{Av74} and \cite[Proposition 3.3.4]{GL}. 
The expression for $\varepsilon_5$ obtained by Avramov in \cite[Corollary 6.2]{Av74} is, in our notations:
$$\varepsilon_5=a_4+a_1a_2+a_1q_{11}-q_{12}-a_1^3+b'-a,$$
where $b':=\rank_\kk(\Ker \Gamma)$ with
\begin{align*}
\Gamma:& A_1\otimes A_1\otimes A_1\xrightarrow{} A_2\otimes A_1\oplus A_1\otimes A_2,\quad\text{defined by}\\
\Gamma&([x],[y],[z])=([x]\wedge[y],[y]\wedge [z]),\quad\text{for all}\ [x],[y],[z]\in A_1.
\end{align*}
The map $\psi$ in \eqref{def psi}, that defines our invariant $b=\rank_\kk(\Coker\psi)$, differs from $\Gamma$ just by its codomain.
Thus, one can relate $b'$ in \cite{Av74} and our $b$ by:
$$ b'=\rank_\kk(\Ker\psi)=a_1^3-2a_1q_{11}+b.$$
Note that in both references \cite{Av74} and \cite{GL} there is a shift in the indexing of the deviations, their $\varepsilon_i$ is our $\varepsilon_{i+1}$.
\end{remark}


\section{An example illustrating the Construction}
\label{sec: example}

In this section we consider the codepth 4 artinian local ring
$$R=\BQ[x,y,z,w]/(x^3, y^3, z^3-xy^2, x^2z^2, xyz^2, y^2w, w^2)$$ from \cite[Section 7]{Av74}.
This ring is of a particular interest to us, since its Koszul homology algebra $A=\HH{}{K^R}$ has nontrivial multiplication and a nontrivial ternary Massey product that does not come from this multiplication.
The free modules $\{F_i\}_{i=0,\dots,5}$ are given in terms of  the Koszul algebra components
$\{K_i\}_{i=0,\dots,5}$, the ranks $a_i$ of $A_i$, and the multiplicative invariants $q_{ij}, a$ and $b$.
The differential maps of the complex $F$ in Construction \ref{construction} are given in terms of elements $z^1_i, z^2_\ell, z^3_t, z^4_r$,  $\widetilde p^1_{si}, \widetilde\pi^3_s, \widetilde p^2_{ui}$, and $\widetilde\pi^4_u$, see Summary~\ref{summary}.
We explicitly describe them all for this ring. The bases of $A_i=\HH{i}{K}$ are computed with {\tt Macaulay2} \cite{M2}.
We use our results from Section \ref{sec multiplication} to obtain the other elements needed in Construction \ref{construction}.

\chunk The ring $R$ is graded artinian with $R_i=0$ for all $i\geq 6 $. The other graded components have the following basis elements:
\begin{align*}
R_0:&\quad 1\\
R_1:&\quad x,\, y,\, z,\, w\\
R_2:&\quad x^2,\,  xy,\,  xz,\,  xw,\,  y^2,\,  yz,\,  yw,\,  z^2, \, zw\\
R_3:&\quad  x^2y,\,  x^2z,\, \,  x^2w,\,  xyz,\,  xyw,\,  xz^2,\,  xzw,\,  y^2z,\,  yz^2,\,  yzw,\,  z^3,\,  z^2w\\
R_4:&\quad x^2yz,\,  x^2yw,\,  x^2zw,\,  xyzw,\,  xz^3,\,  xz^2w,\,  y^2z^2,\,  yz^2w,\,  z^4\\
R_5:&\quad x^2yzw,\,  xz^4.
\end{align*}

\chunk The Koszul algebra of $R$ has the form $K=\bigwedge(R^4)\cong R\oplus R^4\oplus R^6\oplus R^4\oplus R.$
Let $\{T_1,T_2,T_3,T_4\}$ be the standard  ordered vector basis of the free module $K_1=R^4$.
Set $T_{ij}=T_i\wedge T_j$ and consider the  ordered basis
$\{T_{12},T_{13},T_{23}, T_{14}, T_{24}, T_{34}\}$ for the free module $K_2=\bigwedge^2 (R^4)\cong R^6$.
Set $T_{ijk}=T_i\wedge T_j\wedge T_k$ and consider the ordered basis
$\{T_{123},T_{124},T_{134}, T_{234}\}$ for the free module $K_3=\bigwedge^3 (R^4)\cong R^4$.
Set $T_{1234}=T_1\wedge T_2\wedge T_3\wedge T_4$ and consider the basis $\{T_{1234}\}$ for the free module $K_4=\bigwedge^4 (R^4)\cong R$. By {\tt Macaulay2}, we obtain bases of $A_i=\HH{i}{K}$ and  ranks
$$a_1=7,\quad a_2=15,\quad a_3=14,\quad a_4=5.$$

\chunk A basis of $A_1$ is $\{[z_i^1]\}_{i=1,\dots,7}$ where:
\begin{align*}
  z_1^1&=wT_4,&
  z_2^1&=x^2T_1,&
  z_3^1&=ywT_2,&
  z_4^1=y^2T_2,\\
  z_5^1&=y^2T_1-z^2T_3,&
  z_6^1&=yz^2T_1,&
  z_7^1&=xz^2T_1.&
\end{align*}

\chunk Using Koszul and ring relations, we obtain that the space $A_1\cdot A_1$ has rank  $q_{11}=7$ and its basis is given by the classes of
\begin{align*}
z_1^1\wedge z_2^1&=-x^2wT_{14},&
z_1^1\wedge z_5^1&=z^2wT_{34},&
z_1^1\wedge z_6^1&=-yz^2wT_{14},&
z_1^1\wedge z_7^1&=-xz^2wT_{14},\\
z_2^1\wedge z_3^1&=x^2ywT_{12},&
z_3^1\wedge z_5^1&=-yz^2wT_{23},&
z_4^1\wedge z_5^1&=-y^2z^2T_{23}.
\end{align*}
Therefore, a basis of $\overline {A_2}$ is $\{[z_\ell^2]\}_{\ell=1,\dots,8}$, where
\begin{align*}
  z_1^2&=ywT_{24},&
  z_2^2&=y^2T_{24},&
  z_3^2&=xz^2T_{12},&
  z_4^2&=yz^2T_{13},\\
  z_5^2&=x^2yT_{12}-xz^2T_{13},&
  z_6^2&=x^2zT_{13},&
  z_7^2&=yz^2wT_{12},&
  z^2_8&=z^4T_{23}.
\end{align*}
\chunk
\label{p1 ex}
All  other products among the basis elements of $A_1$ are zero in $A_2$, except for the  Koszul relations of the nonzero products $A_1\cdot A_1$ above.
It follows that the kernel of the multiplication map $\phi_1: A_1\otimes A_1\to A_2$ is generated by
$a_1^2-q_{11} =42$ elements.
We record  the indices of basis elements of $A_1\cdot A_1$ by the set $$S=\{(1,2),\ (1,5),\ (1,6),\ (1,7),\ (2,3),\ (3,5),\ (4,5)\}.$$
A basis of $\Ker\phi_1$, as defined in \eqref{def p1}, is given by elements of two types:
\begin{itemize}
\item[(1)] $[z^1_i]\otimes[z^1_j],\,\quad\text{for}\  1\leq i,j\leq 7,\quad (i,j)\not\in S\quad\text{and}\quad (j,i)\not\in S$; \vspace{0.05in}
\item[(2)] $[z^1_i]\otimes[z^1_j]+[z^1_j]\otimes[z^1_i],\,\quad \text{for}\quad (i,j)\in S$.
\end{itemize}
Therefore, according to the types above, one defines
$\widetilde p^1_{s_{(i,j)}}=(\widetilde p^1_{s_{(i,j)}k})_{k=1,\dots,7}$ as:
\begin{itemize}
  \item[(1)] $\widetilde p^1_{s_{(i,j)}k}=
  \begin{cases}z^1_j&\text{if}\quad  k=i\\
  0&\text{if}\quad k \not=i,
  \end{cases}\qquad$
  for all $ 1\leq i,j\leq 7,\quad (i,j)\not\in S\quad\text{and}\quad (j,i)\not\in S$; \vspace{0.05in}
\item[(2)]
$\widetilde p^1_{s_{(i,j)}k}=
\begin{cases}
  z^1_j&\text{if}\quad k=i\\
  z^1_i&\text{if}\quad k=j\\
        0&\text{if}\quad k \not=i,j,
\end{cases}\quad$
for all $(i,j)\in S$.
\end{itemize}
\chunk Next, we find the nonzero elements $\widetilde \pi^3_s\in K_3$ defined in  \eqref{def pi3} as follows. For the elements coming from Koszul relations in $K$, we choose $\widetilde \pi^3_s=0$ for those elements.
The only nonzero products in $K_2$  come from elements of type (1) and they are
$$ z^1_2\wedge z^1_4=x^2y^2T_{12}=-z^1_4\wedge z^1_2\qquad\text{and}\qquad z^1_5\wedge z^1_7=xz^4T_{13}=-z^1_7\wedge z^1_5.$$
Therefore, we choose the following nontrivial liftings in $K_3$:
$$
\widetilde \pi^3_{s_{(2,4)}}= xz^2T_{123}=-\widetilde \pi^3_{s_{(4,2)}}\qquad\text{and}\qquad \widetilde \pi^3_{s_{(5,7)}}= -x^2yzT_{123}=-\widetilde \pi^3_{s_{(7,5)}}.$$

\chunk We describe next the elements of $A_3$. 

\vspace{0.1in}

\noindent
{\bf Claim 1.} All the products in $A_1\cdot A_1\cdot A_1$ are zero.
\begin{proof}[\it Proof of Claim 1.]
First, all the products in $A_1\cdot A_1\cdot A_1$ not involving $[z^1_1]$  have the coefficients of $[T_{ijk}]$ of degree six, and $R_6=0$.  Thus, all such products are zero.
Second, the only pairs $(i,j)$ with $1<i<j\leq 7$ such that $[z^1_i]\wedge[z^1_j]$ is nonzero in $A_1\cdot A_1$ are in the set $\{(2,3), (3,5), (4,5)\}$, but $[z^1_1\wedge z^1_3]=[z^1_1\wedge z^1_4]=0$. Hence, all products involving $[z^1_1]$ at least once are zero. The claim now follows.
\end{proof}

Therefore, $A_1\cdot A_2=A_1\cdot \overline{A_2}$, it has rank $q_{12}=10$, and its basis is given by the classes of
\begin{align*}
z^1_1\wedge z^2_3&=xz^2wT_{124},&
z^1_1\wedge z^2_4&=yz^2wT_{134},&
z^1_1\wedge z^2_5&=x^2ywT_{124}-xz^2wT_{134},\\
z^1_1\wedge z^2_6&=x^2zwT_{134},&
z^1_2\wedge z^2_1&=x^2ywT_{124},&
z^1_2\wedge z^2_2&=x^2y^2T_{124},\\
z^1_3\wedge z^2_6&=-x^2yzwT_{123},&
z^1_5\wedge z^2_1&=yz^2wT_{234},&
z^1_5\wedge z^2_2&=y^2z^2T_{234},\\
z^1_5\wedge z^2_3&=-xz^4T_{123}=z^1_4\wedge z^2_6.
\end{align*}
A basis of $\overline {A_3}$ is $\{[z_t^3]\}_{t=1,\dots,4}$, where
\begin{align*}
  z_1^3&=yz^2wT_{123},&
  z_2^3&=y^2z^2T_{123},&
  z_3^3&=yz^2wT_{124},&
  z_4^3&=z^4T_{234}.
\end{align*}

\chunk
Using Proposition $\ref{ker2}$, we describe the elements $\widetilde p^2_{ui}$,  for $1\leq u\leq a_1a_2-a_1q_{11}-q_{12}+b=46+b$ and $1\leq i\leq 7$, as defined in \eqref{def p2}. 

\vspace{0.1in}

\noindent
{\bf Claim 2.} Let $\CA$, $\CB$ be as in Proposition \ref{ker2} and let $b$ be as in Summary~\ref{summary}.
Then
$$ \Span_\kk \CA=A_1\otimes(A_1\cdot A_1)\quad\text{and}\quad b=0.$$

\begin{proof}[Proof of Claim 2] It is clear that $\Span_\kk \CA\subseteq A_1\otimes(A_1\cdot A_1)$.
For $1\leq i,j,h\leq 7$, any nonzero element  $[z^1_i]\otimes[z^1_j]\wedge[z^1_h]$ in $A_1\otimes(A_1\cdot A_1)$ has $(j,h)\in S$ or $(h,j)\in S.$ We show that each such element is in $\Span_\kk\CA.$

In the case $\{(i,j), (i,h)\}\not\subseteq S$ we have
$$[z^1_i]\otimes[z^1_j]\wedge[z^1_h]=\sum_{k=1}^{7}[z^1_k]\otimes[\widetilde p^1_{s_{(i,j)}k}]\wedge[z^1_h],$$
as in case (1) of \ref{p1 ex}.

By Koszul relation, the case $(i,h)\in S$ reduces to the case $(i,j)\in S$ in which we have
$$[z^1_i]\otimes[z^1_j]\wedge[z^1_h]=[z^1_i]\otimes[z^1_j]\wedge[z^1_h]+[z^1_j]\otimes[z^1_i]\wedge[z^1_h]=\sum_{k=1}^{7}[z^1_k]\otimes[\widetilde p^1_{s_{(i,j)}k}]\wedge[z^1_h],$$
as in case (2) of \ref{p1 ex}. The first equality above follows from the proof of Claim 1.
This implies  $\Span_\kk \CA=A_1\otimes(A_1\cdot A_1)$, so  $\rank_\kk(\Span_\kk \CA)=a_1q_{11}$. By Proposition \ref{ker2}(b) we get $b=0$.
\end{proof}

Remark that Claim 2 implies Claim 1, since $\Span_\kk\CA\subseteq\Ker\phi_2$.

By definition of $\overline{A_2}$ as in Summary~\ref{summary} and respectively Claim 2, the following equalities hold:
 $$A_1\otimes A_2=\big(A_1\otimes (A_1\cdot A_1)\big)\oplus (A_1\otimes \overline{A_2})=(\Span_\kk \CA)\oplus (A_1\otimes \overline{A_2}).$$
By Proposition \ref{ker2},  $\Ker\phi_2=(\Span_\kk \CA)\oplus B$, where $B=(\Ker\phi_2)\cap (A_1\otimes \overline{A_2})$. In order to find a basis for $B$, we record the indices of basis elements of $A_1\cdot \overline{A_2}$ by the set
$$U=\{(1,3), (1,4), (1,5), (1,6), (2,1), (2,2), (3,6), (5,1), (5,2), (5,3),(4,6)\}.$$
A basis of $B$ is given by elements of two types:
\begin{itemize}
\item[$(1')$] $[z^1_i]\otimes[z^2_\ell]$ for all $1\leq i\leq 7$ and $1\leq\ell\leq 8$ such that $(i,\ell)\not\in U$; \vspace{0.05in}

\item[$(2')$] $[z^1_4]\otimes[z^2_6]-[z^1_5]\otimes[z^2_3]$.
\end{itemize}
Therefore, according to the types above, one defines $\widetilde p^2_{u_{(i,\ell)}}=(\widetilde p^2_{u_{(i,\ell)}k})_{k=1,\dots,7}$ as:
\begin{itemize}
\item[$(1')$] $\widetilde p^2_{u_{(i,\ell)}k}=
\begin{cases}
  z^2_\ell&\ \text{if}\ k=i\\
        0& \ \text{if}\ k\not=i,
\end{cases} $ for all $1\leq i\leq 7$ and $1\leq\ell\leq 8$ such that $(i,\ell)\not\in U$; \vspace{0.05in}

\item[$(2')$] $\widetilde p^2_{u_{(4,6)}k}=
\begin{cases}
  z^2_6&\ \text{if}\ k=4\\
  -z^2_3&\ \text{if}\ k=5\\
  0& \ \text{if}\ k\not=4,5.
\end{cases} $
\end{itemize}
\chunk It is easy to check that in $K_3$  we have the equalities:
\begin{itemize}
\item[$(1')$] $z^1_i\wedge z^2_\ell=0$ for all $1\leq i\leq 7$ and $1\leq\ell\leq 8$ such that $(i,\ell)\not\in U$; \vspace{0.05in}

\item[$(2')$] $z^1_4\wedge z^2_6 - z^1_5\wedge z^2_3=0$.
\end{itemize}
Therefore, we may choose $\widetilde \pi^4_u=0$ for all $1\leq u\leq 46.$

\chunk A similar argument as in the proof of Claim 1  gives  $A_1\cdot A_3=0$, and hence $q_{13}=0$.
Moreover, $A_2\cdot A_2$ has rank $q_{22}=2$ and its basis is given by the classes of
$$
  z^2_1\wedge z^2_6=-x^2yzwT_{1234}\qquad\text{and}\qquad z^2_2\wedge z^2_6=-xz^4T_{1234}.
$$
As $A_1\cdot A_3=0$, any element in  $\langle A_1, A_1, A_1\rangle\subseteq A_4$ has a representative in $K_4$ given by
\begin{equation}
  \label{massey eq}
\sum_{s=1}^{42}\widetilde\pi^3_s\wedge  p^1_s
\quad
\text{such that}
\quad
\sum_{s=1}^{42}[\widetilde p^1_{sk}\wedge p^1_s]=0,\ \text{for all}\ 1\leq k\leq 7,
\end{equation}
where $p^1_s\in\Ker\partial_1^K$ is as in Proposition \ref{massey}, and $\widetilde p^1_{si}$ and  $\widetilde \pi^3_{s}$ are  defined in  \eqref{def p1} and \eqref{def pi3} respectively. As described above, the only nontrivial lifting elements $\widetilde\pi^3_s \in K_3$ are
$\widetilde\pi^3_{s_{(2,4)}}=-\widetilde\pi^3_{s_{(4,2)}}$ and
$\widetilde\pi^3_{s_{(5,7)}}=-\widetilde\pi^3_{s_{(7,5)}}$.
Since all of them contain $T_{123}$, only the component of $p^1_s$ that contains $z^1_1=wT_4$ contributes to a nonzero Massey product.
Thus, there exist $\alpha,\beta\in R$ such that
\begin{align*}
  \sum_{s=1}^{42}[\widetilde\pi^3_s\wedge  p^1_s] &=[\widetilde\pi^3_{s_{(2,4)}}\wedge p^1_{s_{(2,4)}}]+[\widetilde\pi^3_{s_{(5,7)}}\wedge p^1_{s_{(5,7)}}]\\
  &=[\widetilde\pi^3_{s_{(2,4)}}\wedge z^1_1]\wedge[\alpha]+[\widetilde\pi^3_{s_{(5,7)}}\wedge z^1_1]\wedge[\beta]\\
  &=[xz^2wT_{1234}]\wedge[\alpha]+[x^2yzw T_{1234}]\wedge[\beta]\\
  &=\langle [z^1_2],[z^1_4],[z^1_1]\rangle\wedge[\alpha]-[z^2_1\wedge z^2_6]\wedge[\beta].
\end{align*}
The last equality follows from the following computation:
\begin{align*}
  \langle [z^1_2],[z^1_4],[z^1_1]\rangle&=\{[(\partial_3^K)^{-1}(z^1_2 \wedge z^1_4)\wedge z^1_1 + z^1_2 \wedge (\partial_3^K)^{-1}(z^1_4 \wedge z^1_1)]\}\\
  &=\{[\widetilde\pi^3_{s_{(2,4)}}\wedge z^1_1 + z^1_2 \wedge \widetilde\pi^3_{s_{(4,1)}}]\}\\
  &=\{[xz^2wT_{1234}]\}.
\end{align*}
By abusing notation, we write  $\langle [z^1_2],[z^1_4],[z^1_1]\rangle=[xz^2wT_{1234}]$, which is not in $A_1\cdot A_3+A_2\cdot A_2$, as showed in \cite[Section 7]{Av74}.  It is clear now that the rank of the $\kk$-vector space  
$$A_1\cdot A_3+A_2\cdot A_2+\text{Span}_\kk\langle A_1\cdot A_1\cdot A_1\rangle$$ 
is $a=3$ and its basis is given by
$$[z^2_1]\wedge [z^2_6],\quad [z^2_2]\wedge [z^2_6],\quad \text{and}\quad \langle [z^1_2],[z^1_4],[z^1_1]\rangle.$$
Thus, a basis of $\overline{A_4}$ is $\{[z_r^4]\}_{r=1,2}$  where
$$
  z_1^4=yz^2wT_{1234}\qquad\text{and}\qquad  z_2^4=y^2z^2T_{1234}.
$$

\chunk We conclude that for the ring $R$ discussed in this section, and by Proposition \ref{prop:P} we have:
\begin{align*}
  a_1&=7, &a_2&=15, &a_3&=14, &a_4&=5,\\
  q_{11}&=7, &q_{12}&=10, &q_{13}&=0, &q_{22}&=2,\\
  a&=3, &b&=0, &&&& \\
  \CP_3&=7, &\CP_{4}&=45, &\CP_5&=221. &&
\end{align*}

 \section*{Acknowledgement}
 The authors thank Hailong Dao for inspiring them to work on this project and for very fruitful discussions.
 They also thank Frank Moore for helping them with the {\tt DGAlgebras} package \cite{DGA} to check the computations, and thank the referee for helpful suggestions. 
 The first author was supported by the Naval Academy Research Council in Summer 2020.


\end{document}